\documentclass[review,3p,times,sort&compress]{elsarticle}
\usepackage{amssymb}
\usepackage{tikz}
\usepackage{graphicx,fancyhdr}
\usepackage{graphicx}
\usepackage{subcaption}
\usepackage{multirow}
\usepackage{dsfont}
\usepackage{bbm}
\usepackage{booktabs}
\usepackage{amsthm}
\usepackage{bbm,bm}
\DeclareMathAlphabet{\mathpzc}{OT1}{pzc}{m}{it}
\usepackage{bbm}
\usepackage{dsfont}
\usepackage{yfonts}
\usepackage{amsmath}
\usepackage{amssymb}
\usepackage{amsfonts}
\usepackage{mathrsfs,float}
\usepackage{color}
\usepackage[colorlinks, citecolor=blue]{hyperref}

\makeatletter

\journal{}

\begin{document}
	
	\begin{frontmatter}
		
		
		\newtheorem{theorem}{Theorem}[section]
		\newtheorem{remark}[theorem]{Remark}
		\newtheorem{ex}{Example}
		\newtheorem{case}{Case}
		\newtheorem{pro}[theorem]{Proposition}
		\newtheorem{defi}[theorem]{Definition}
		\newtheorem{ass}{Assumption}
		\newtheorem{lemma}[theorem]{Lemma}
	\newtheorem{corollary}[theorem]{Corollary}
\newtheorem{assumption}[theorem]{Assumption}
		\newproof{pf}{Proof}
		\newproof{pot}{Proof of Theorem \ref{thm2}}
		\title{A Patchwise Local Fourier Extension Method for Function Approximation on General Two-Dimensional Domains}
			\author[1]{Zhenyu Zhao}
	\author[2]{Yanfei Wang}
		\address[1]{School of Mathematics and Statistics, Shandong University of Technology, Zibo, 255049, China}
\address[2]{Key Laboratory of Deep Petroleum Intelligent Exploration and Development,Institute of Geology and Geophysics, Chinese Academy of Sciences, Beijing, 100029, China}
\begin{abstract}
We propose a patchwise local Fourier extension method for approximating
smooth functions on general two dimensional domains with curved boundaries.
The domain is embedded into a Cartesian background grid and decomposed into
rectangular interior patches and one-side curved trapezoidal boundary patches.
After local data transfer, all patches are converted into fixed-size
tensor-product arrays and approximated by a truncated-SVD stabilized local
Fourier extension procedure.

Unlike global Fourier frame approximations, the proposed method localizes
both the geometry and the ill-conditioned extension process. For fixed local
parameters, the local algebraic operations are performed on fixed-size
systems, and the reference Fourier extension matrices and their singular value
decompositions are reused across patches. Boundary patches require additional
one-dimensional transfer or completion steps, but their costs remain uniformly
bounded by the local resolution. Consequently, the online complexity is
\(O(N)\), where \(N\) denotes the total number of retained output points for
fixed local resolution.

Numerical experiments on smooth curved domains and on a mildly rough boundary
domain demonstrate that the method achieves high accuracy with a fixed set of
local parameters. The smooth-cover correction reduces the boundary-induced
error by several orders of magnitude in the full-domain rough-boundary test,
without changing the underlying scan-based partition.
\end{abstract}
\begin{keyword}
Fourier extension;
local approximation;
function approximation on curved domains;
patchwise method;
curved boundary patches;
TSVD regularization;
scan-based partition
\end{keyword}
		
	\end{frontmatter}
	
\section{Introduction}

High-order approximation of smooth functions on domains with curved
boundaries is a fundamental problem in scientific computing and data
science, e.g., mapping the irregular geometry of active Earth faults, mitigating boundary-induced errors at fault intersections, performing data assimilation in the local Fourier coefficient space rather than the physical grid space, and hence  providing an potential spatial representation tool for complex physical  modelling \cite{Yanfei2011}.
More broadly, it is closely related to spectral and high-order discretizations
of partial differential equations, embedded and fictitious-domain methods,
inverse problems, imaging, and data processing on non-rectangular spatial
regions
\cite{Boyd2002,Boyd2005,BrunoHanPohlman2007,
BrunoLyon2010,LyonBruno2010,MatthysenHuybrechs2018,
AdcockHuybrechs2020}. On tensor-product domains, Fourier and polynomial spectral methods are known
to provide highly accurate and efficient approximations. However, this
advantage is much less straightforward to retain when the physical domain has
a non-rectangular or curved boundary. In this case, the boundary geometry
becomes a central numerical difficulty: standard tensor-product
approximations do not conform naturally to the physical boundary, while
unstructured or mesh-based approaches may introduce additional complications
related to local approximation order, conditioning, and the accurate
representation of the geometry.

A common strategy for approximating functions on non-tensor-product domains
is to embed the physical domain \(\Omega\) into a simple tensor-product
domain \(D\), and to restrict an approximation system on \(D\) to
\(\Omega\). This leads naturally to frame approximations rather than basis
approximations. In the polynomial setting, Adcock and Huybrechs studied
polynomial frame approximation for smooth multivariate functions on
irregular domains, using tensor-product orthogonal polynomials on a bounding
box and samples taken only inside the physical domain
\cite{AdcockHuybrechs2020}. In the Fourier setting, Fourier extension (FE),
or Fourier continuation (FC), approximates nonperiodic functions by Fourier
series defined on an enlarged periodic domain
\cite{Boyd2002,Huybrechs2010,AdcockHuybrechsMartinVaquero2014}. In these
methods, the restricted approximation system is generally redundant and the
associated discrete systems are ill-conditioned. Nevertheless, with
regularization, especially truncated singular value decomposition (TSVD),
stable and accurate approximations can be obtained
\cite{AdcockHuybrechsMartinVaquero2014,AdcockHuybrechs2011,
AdcockHuybrechs2019}.

For two-dimensional irregular domains, Matthysen and Huybrechs developed a
global Fourier frame approximation by restricting a Fourier basis on a
bounding box to samples inside the physical domain
\cite{MatthysenHuybrechs2018}. Their analysis shows that the singular values
of the corresponding collocation matrix cluster near one and zero, with a
relatively small transition set, or plunge region, between them. In two
dimensions, the size of this plunge region is governed mainly by the boundary
geometry rather than by the area of the domain. Consequently, the complexity
of their fast projection algorithm is closely tied to the boundary measure
and, for ordinary non-fractal planar domains, can be reduced from the cubic
cost of a direct solver to \(O(N^2\log^2 N)\)
\cite{MatthysenHuybrechs2018}. This observation is important here: it
suggests that, in Fourier-frame approximation on irregular domains, the main
geometric difficulty is concentrated near the boundary.

Another important line of work is based on explicit continuation near the
boundary. Bruno and coauthors developed high-order Fourier-continuation
techniques for solving partial differential equations in general smooth
domains
\cite{BrunoHanPohlman2007,BrunoLyon2010,LyonBruno2010,AlbinBruno2011}.
More recently, Bruno and Paul proposed a two-dimensional Fourier
Continuation method for constructing biperiodic extensions of smooth
functions on general smooth domains \cite{BrunoPaul2022}. In such approaches,
the main task is to construct smooth continuation data near the boundary so
that the function can be embedded into a periodic rectangular setting and
then processed efficiently by Fourier techniques such as the FFT. These
methods demonstrate the effectiveness of combining Fourier approximation with
careful boundary continuation, but they also show that the treatment of the
boundary is a central and delicate part of Fourier-based methods on curved
domains.

The discussion above suggests two complementary ways of handling irregular
geometries. Global frame methods restrict a tensor-product approximation
system to the physical domain and control the resulting redundancy by
regularization, whereas Fourier-continuation methods construct extension data
so that a periodic representation can be used on a surrounding rectangular
domain. In both cases, the geometric difficulty is strongly tied to the
boundary. This motivates a local alternative: separate the relatively simple
interior regions from the boundary regions, approximate each local piece by a
fixed-size stabilized Fourier extension procedure, and assemble the resulting
local approximations.

Localized Fourier extension strategies provide such a route. In our previous
work \cite{ZhaoWang2026}, a one-dimensional local Fourier extension (LFE) method
was developed for function approximation on intervals. The method partitions
the target interval into subintervals and applies a stabilized Fourier
extension approximation on each local piece. It retains the high accuracy of
Fourier extensions for smooth functions while adapting to local variations of
the target function. The local construction is also modular: once the local
discretization parameters are fixed, the same reference extension matrix and
its truncated singular value decomposition can be reused.

The present paper extends this LFE philosophy to two-dimensional domains
with curved boundaries. The main new difficulty is geometric: a local
partition must handle not only rectangular interior regions, but also curved
boundary regions and transition portions where the dominant boundary
direction changes. We therefore construct a patch system that is compatible
with Cartesian sampling and local tensor-product Fourier extension.

The proposed method embeds the domain into a Cartesian background grid and
uses a boundary scan to generate a local patch database. Interior regions are
treated as rectangular patches, while boundary regions are represented by
one-side curved patches of left, right, top, or bottom type. Transition
regions are covered by additional one-side patches and reduced to the same
four standard boundary types. On rectangular patches, tensor-product LFE is applied directly. On curved patches, boundary
intersections with vertical or horizontal sampling lines are used to form a
fixed-size local data array, so that the same tensor-product LFE
procedure can be used. Thus the geometric information is confined to the
patch construction and data transfer, while the algebraic approximation step
remains uniform across all patches.

This locality gives the method a clear computational advantage. Instead of
forming a single global Fourier frame system, the approximation is decomposed
into local problems of fixed resolution. For rectangular patches, the
tensor-product LFE procedure uses precomputed one-dimensional SVDs directly.
For curved boundary patches, additional one-dimensional LFE transfers are
performed along vertical or horizontal sampling lines; for mildly rough
boundary patches, a one-unknown completion is also carried out on each
sampling line. Since all these operations involve only a fixed number of local
nodes when \(m,n,T\) and the refinement factor are fixed, the online cost is
proportional to the number of patches, or equivalently \(O(N)\) with respect
to the number of retained output points.

The main contributions of this paper are summarized as follows.
\begin{itemize}
\item We develop a patchwise local Fourier extension framework for
two-dimensional domains with curved boundaries. The scan-based construction
reduces the domain to rectangular interior patches and four standard types of
one-side curved boundary patches.

\item We show how curved boundary data can be converted into fixed-size
tensor-product arrays by one-dimensional local Fourier extension transfer.
Consequently, all local patches are processed by the same TSVD-stabilized
tensor-product Fourier extension procedure.

\item We establish a basic error estimate that relates the two-dimensional
local approximation error to one-dimensional directional Fourier extension
defects and to the geometry of the local patch mapping. The estimate clarifies
why local subdivision reduces the effective frequency and the geometric
distortion seen by each patch.

\item We analyze the computational complexity of the method. For fixed local
parameters, the local algebraic operations are performed on fixed-size
systems with reusable SVDs, leading to \(O(N)\) online complexity with respect
to the total number of retained output points.

\item We introduce a smooth-cover correction for boundary patches with
small-scale rough perturbations. The correction constructs a local smooth
cover, recovers one artificial boundary value on each sampling line using a
one-dimensional Fourier extension consistency condition, and then restricts
the final output to the original physical region.
\end{itemize}

Numerical experiments on both smooth and mildly rough domains demonstrate
that the method achieves high accuracy with a fixed set of local parameters.
The full-domain rough-boundary tests further show that the smooth-cover
correction can substantially reduce boundary-induced errors without modifying
the scan-based partition.

The rest of the paper is organized as follows. Section~2 formulates the
patchwise local Fourier extension problem on rectangular and one-side curved
patches. Section~3 gives a basic error estimate and discusses the influence
of local subdivision and geometric mapping. Section~4 calibrates the main
numerical parameters on rectangular patches. Section~5 describes the
approximation procedure on curved trapezoidal patches, including the
smooth-cover correction for mildly rough boundary patches. Section~6 presents
the scan-based patch construction for general domains and analyzes the
computational complexity. Section~7 reports numerical experiments, including
tests on smooth curved domains and on a domain with a mildly rough boundary.
Section~8 concludes the paper with several remarks on possible extensions.

\section{Problem formulation}
\label{sec:problem-formulation}

In this section, we formulate the two-dimensional patchwise local Fourier
extension problem. The aim is to approximate a smooth function on a general
planar domain by reducing the computation to local Fourier extension
approximations on geometrically simple patches.

Let \(\Omega\subset\mathbb{R}^2\) be a bounded planar domain whose boundary
\(\Gamma=\partial\Omega\) can be locally represented by curves. We assume that
\(\Omega\) is contained in a rectangular box
\[
    R=[x_{\min},x_{\max}]\times[y_{\min},y_{\max}],
\]
equipped with a Cartesian background partition. The physical domain is
decomposed, or covered, by a finite collection of local patches
\[
    \mathcal P=\{P_\nu\}_{\nu=1}^{N_p}.
\]
The patches are either rectangular interior patches or boundary-adapted
one-side curved patches. This decomposition reduces approximation on a
general curved domain to a family of local approximation problems for which
tensor-product Fourier extension can be used.

For an interior patch, we take
\[
    P_\nu=[x_L,x_R]\times[y_B,y_T].
\]
For a boundary patch, we assume that, after choosing a suitable local
coordinate direction, the patch has one curved side and three Cartesian
aligned sides. Typical examples are
\[
    P_T=\{(x,y): x_L\le x\le x_R,\ y_B\le y\le \gamma_T(x)\},
\]
\[
    P_B=\{(x,y): x_L\le x\le x_R,\ \gamma_B(x)\le y\le y_T\},
\]
and, after exchanging the roles of \(x\) and \(y\),
\[
    P_L=\{(x,y): y_B\le y\le y_T,\ \gamma_L(y)\le x\le x_R\},
\]
\[
    P_R=\{(x,y): y_B\le y\le y_T,\ x_L\le x\le \gamma_R(y)\}.
\]
Here \(\gamma_T,\gamma_B,\gamma_L,\gamma_R\) denote local representations of
the boundary. Thus each boundary patch contains only one curved side. This
form is convenient for constructing local tensor-product data arrays from
samples along vertical or horizontal Cartesian lines.

We next introduce the local Fourier extension spaces. Let \(T_x>1\) and
\(T_y>1\) be extension parameters, and set
\[
    \Lambda_x=\left[0,\frac{2\pi}{T_x}\right],
    \qquad
    \Lambda_y=\left[0,\frac{2\pi}{T_y}\right].
\]
For nonnegative integers \(N_x\) and \(N_y\), define
\[
    H_{N_x}^{(x)}
    =
    \mathrm{span}\{e^{i\ell t}:|\ell|\le N_x\},
    \qquad
    H_{N_y}^{(y)}
    =
    \mathrm{span}\{e^{iqs}:|q|\le N_y\}.
\]
The tensor-product local Fourier extension space is
\[
    H_{N_x,N_y}
    =
    H_{N_x}^{(x)}\otimes H_{N_y}^{(y)}
    =
    \mathrm{span}\{
        e^{i\ell t}e^{iqs}:
        |\ell|\le N_x,\ |q|\le N_y
    \}.
\]

On each patch \(P_\nu\), the sampled function values are arranged into a
local data matrix
\[
    \mathbf F^{(\nu)}\in\mathbb C^{m_y\times m_x}.
\]
For rectangular patches, this is the usual tensor-product array in the
physical coordinates. For boundary patches, the data are collected along
Cartesian sampling lines and then transferred to a fixed-size tensor-product
array. For example, on a top patch \(P_T\), we use vertical sampling lines
\(x=x_k\), compute the boundary intersections \(\gamma_T(x_k)\), and sample
the function along
\[
    y_B\le y\le \gamma_T(x_k).
\]
The resulting line data are converted to values on a normalized tensor grid
associated with the local coordinates of the curved patch. The bottom, left,
and right patch types are treated in the same way, with the appropriate
choice of sampling direction.

After rescaling the local coordinates to
\(\Lambda_x\times\Lambda_y\), the local approximant takes the form
\[
    F_\nu(t,s)
    =
    \sum_{|\ell|\le N_x}\sum_{|q|\le N_y}
    c_{\ell,q}^{(\nu)} e^{i\ell t}e^{iqs}.
\]
The coefficients are computed from the local data by a tensor-product
least-squares procedure regularized by truncated singular value decomposition
(TSVD). In practice, this is implemented by applying the one-dimensional
TSVD-stabilized Fourier extension operators successively in the two
coordinate directions.

The global approximation is defined patchwise. The local approximations
\(F_\nu\) are evaluated on their corresponding physical patches and then
assembled to approximate \(f\) on \(\Omega\). The construction of the patch
database, the treatment of transition regions, and the practical assembly
rule are described in the following sections. The basic idea is therefore to transform the approximation problem on a curved domain into a set of local Fourier extension approximations on rectangular patches and one-side curved patches.

\section{A basic error estimate}
\label{sec:error-estimate}

In this section, we give a basic error estimate for the proposed patchwise
local Fourier extension approximation. Starting from the one-dimensional
TSVD-stabilized Fourier extension estimate, we show how the directional
errors enter the tensor-product two-dimensional approximation and how the
local map from a reference patch to a physical patch affects the final error.

\subsection{One-dimensional stabilized estimate}

Let
\[
\mathcal{Q}^{\varepsilon,(x)}_{N_x}:L^2(\Lambda_x)\to H_{N_x}^{(x)},
\qquad
\mathcal{Q}^{\varepsilon,(y)}_{N_y}:L^2(\Lambda_y)\to H_{N_y}^{(y)}
\]
denote the one-dimensional TSVD-stabilized Fourier extension operators in
the two coordinate directions. Here \(\varepsilon>0\) is the truncation
threshold. Following the one-dimensional local Fourier extension theory, we
use the estimates
\begin{equation}
\label{eq:1d-estimate-x}
\|v-\mathcal{Q}^{\varepsilon,(x)}_{N_x}v\|_{L^2(\Lambda_x)}
\le
\inf_{\mathbf c\in\mathbb C^{2N_x+1}}
\left(
\|v-\mathcal{F}^{(x)}_{N_x}\mathbf c\|_{L^2(\Lambda_x)}
+
\varepsilon \|\mathbf c\|_{\ell^2}
\right),
\qquad v\in L^2(\Lambda_x),
\end{equation}
and
\begin{equation}
\label{eq:1d-estimate-y}
\|w-\mathcal{Q}^{\varepsilon,(y)}_{N_y}w\|_{L^2(\Lambda_y)}
\le
\inf_{\mathbf d\in\mathbb C^{2N_y+1}}
\left(
\|w-\mathcal{F}^{(y)}_{N_y}\mathbf d\|_{L^2(\Lambda_y)}
+
\varepsilon \|\mathbf d\|_{\ell^2}
\right),
\qquad w\in L^2(\Lambda_y).
\end{equation}
These estimates state that the stabilized one-dimensional reconstruction is
controlled by the best Fourier extension approximation error together with
the TSVD regularization term.

\subsection{Local estimate on a reference patch}

Let
\[
\Lambda=\Lambda_x\times \Lambda_y .
\]
For \(g\in L^2(\Lambda)\), define the tensor-product local Fourier extension
operator
\[
\mathcal{Q}^{\varepsilon}_{N_x,N_y}
=
\mathcal{Q}^{\varepsilon,(x)}_{N_x}
\otimes
\mathcal{Q}^{\varepsilon,(y)}_{N_y}.
\]
This corresponds to applying the one-dimensional stabilized reconstructions
successively in the two coordinate directions.

For later use, define the directional reconstruction errors
\[
\mathcal E_x(g)
=
\left(
\int_{\Lambda_y}
\|g(\cdot,s)
-
\mathcal{Q}^{\varepsilon,(x)}_{N_x}g(\cdot,s)\|_{L^2(\Lambda_x)}^2
\,ds
\right)^{1/2},
\]
and
\[
\mathcal E_y(g)
=
\left(
\int_{\Lambda_x}
\|g(t,\cdot)
-
\mathcal{Q}^{\varepsilon,(y)}_{N_y}g(t,\cdot)\|_{L^2(\Lambda_y)}^2
\,dt
\right)^{1/2}.
\]

\begin{lemma}
\label{lem:reference-patch-estimate}
Assume that the stabilized one-dimensional reconstruction in the
\(x\)-direction is bounded on \(L^2(\Lambda_x)\), namely
\[
\|\mathcal{Q}^{\varepsilon,(x)}_{N_x}\|_{L^2(\Lambda_x)\to L^2(\Lambda_x)}
\le \kappa_x .
\]
Then
\begin{equation}
\label{eq:reference-patch-estimate}
\|g-\mathcal{Q}^{\varepsilon}_{N_x,N_y}g\|_{L^2(\Lambda)}
\le
\mathcal E_x(g)+\kappa_x\mathcal E_y(g).
\end{equation}
\end{lemma}

\begin{proof}
Using the identity
\[
I-\mathcal{Q}^{\varepsilon}_{N_x,N_y}
=
(I-\mathcal{Q}^{\varepsilon,(x)}_{N_x})\otimes I
+
\mathcal{Q}^{\varepsilon,(x)}_{N_x}
\otimes
(I-\mathcal{Q}^{\varepsilon,(y)}_{N_y}),
\]
we apply the triangle inequality. The first term gives \(\mathcal E_x(g)\)
by Fubini's theorem, while the second term is bounded by
\(\kappa_x\mathcal E_y(g)\). This proves
\eqref{eq:reference-patch-estimate}.
\end{proof}

Combining Lemma~\ref{lem:reference-patch-estimate} with
\eqref{eq:1d-estimate-x}--\eqref{eq:1d-estimate-y}, the two-dimensional
reference-patch error can be controlled by the one-dimensional Fourier
extension approximation defects in the two coordinate directions. Define
\[
\mathcal D_x(g)
=
\left(
\int_{\Lambda_y}
\inf_{\mathbf c(s)}
\left(
\|g(\cdot,s)-\mathcal{F}^{(x)}_{N_x}\mathbf c(s)\|_{L^2(\Lambda_x)}
+
\varepsilon \|\mathbf c(s)\|_{\ell^2}
\right)^2
\,ds
\right)^{1/2},
\]
and
\[
\mathcal D_y(g)
=
\left(
\int_{\Lambda_x}
\inf_{\mathbf d(t)}
\left(
\|g(t,\cdot)-\mathcal{F}^{(y)}_{N_y}\mathbf d(t)\|_{L^2(\Lambda_y)}
+
\varepsilon \|\mathbf d(t)\|_{\ell^2}
\right)^2
\,dt
\right)^{1/2}.
\]
Then
\begin{equation}
\label{eq:reference-best-estimate}
\|g-\mathcal{Q}^{\varepsilon}_{N_x,N_y}g\|_{L^2(\Lambda)}
\le
\mathcal D_x(g)
+
\kappa_x\mathcal D_y(g).
\end{equation}

\begin{remark}[Local subdivision and effective frequency]
\label{rem:local-frequency-reduction}
The quantities
\(\mathcal D_x(g)\) and
\(\mathcal D_y(g)\) measure the one-dimensional Fourier
extension approximation defects of the pulled-back function \(g\) in the two
coordinate directions. If \(g\) is smooth and has moderate local oscillation,
then the residual terms in these quantities are small for suitable
coefficient vectors \(\mathbf c(s)\) and \(\mathbf d(t)\) with moderate
\(\ell^2\)-norms.

The role of the patch partition is analogous to that in the one-dimensional
local Fourier extension method. If the physical function contains an
oscillatory component of frequency \(\omega_x\) in the \(x\)-direction, then
on a patch of width \(\Delta x\) its effective local frequency after
rescaling is proportional to \(\omega_x\Delta x\). Therefore, refining the
partition reduces the local frequency seen by the Fourier extension
approximation. This explains why fixed moderate Fourier orders can remain
effective on sufficiently small patches.
\end{remark}

\subsection{Patchwise estimate on the physical domain}

Let \(P\) be a local physical patch, either rectangular or one-side curved,
and let
\[
\Psi_P:\widehat Q\to P
\]
be the corresponding map from a fixed reference square \(\widehat Q\). After
an affine rescaling from \(\widehat Q\) to
\(\Lambda=\Lambda_x\times\Lambda_y\), we still denote the pulled-back function
by
\[
g_P=f\circ \Psi_P.
\]
Assume that \(\Psi_P\) is bi-Lipschitz and that
\[
0<J_{P,\min}
\le
|\det D\Psi_P|
\le
J_{P,\max}
<\infty.
\]
Let
\[
G_P=\mathcal{Q}^{\varepsilon}_{N_x,N_y}g_P
\]
be the local Fourier extension approximation on the reference patch, and
define
\[
F_P(x,y)=G_P(\Psi_P^{-1}(x,y)),
\qquad (x,y)\in P.
\]

\begin{theorem}
\label{thm:physical-patch-estimate}
Under the assumptions above,
\begin{equation}
\label{eq:physical-patch-estimate}
\|f-F_P\|_{L^2(P)}
\le
J_{P,\max}^{1/2}
\left(
\mathcal E_x(g_P)+\kappa_x\mathcal E_y(g_P)
\right).
\end{equation}
Equivalently, in terms of the directional best-approximation defects,
\begin{equation}
\label{eq:physical-patch-best-estimate}
\|f-F_P\|_{L^2(P)}
\le
J_{P,\max}^{1/2}
\left(
\mathcal D_x(g_P)
+
\kappa_x\mathcal D_y(g_P)
\right).
\end{equation}
\end{theorem}

\begin{proof}
By the change of variables \((x,y)=\Psi_P(\xi,\eta)\), we have
\[
\|f-F_P\|_{L^2(P)}^2
=
\int_{\widehat Q}
|g_P(\xi,\eta)-G_P(\xi,\eta)|^2
|\det D\Psi_P(\xi,\eta)|
\,d\xi d\eta .
\]
Using the upper bound on the Jacobian determinant gives
\[
\|f-F_P\|_{L^2(P)}
\le
J_{P,\max}^{1/2}
\|g_P-G_P\|_{L^2(\widehat Q)}.
\]
After accounting for the affine rescaling between \(\widehat Q\) and
\(\Lambda\), Lemma~\ref{lem:reference-patch-estimate} yields
\eqref{eq:physical-patch-estimate}. Estimate
\eqref{eq:physical-patch-best-estimate} follows from
\eqref{eq:reference-best-estimate}.
\end{proof}

\subsection{Global patchwise estimate}

Assume that the final physical output is assigned on a collection of patches
\(\mathcal P\). If the output patches have disjoint interiors, then
\[
\|f-F\|_{L^2(\Omega)}^2
=
\sum_{P\in\mathcal P}
\|f-F_P\|_{L^2(P)}^2.
\]
Therefore, by Theorem~\ref{thm:physical-patch-estimate},
\begin{equation}
\label{eq:global-patchwise-estimate}
\|f-F\|_{L^2(\Omega)}
\le
\left(
\sum_{P\in\mathcal P}
J_{P,\max}
\left(
\mathcal E_x(g_P)+\kappa_x\mathcal E_y(g_P)
\right)^2
\right)^{1/2}.
\end{equation}
Using the directional best-approximation form, we also obtain
\begin{equation}
\label{eq:global-patchwise-best-estimate}
\|f-F\|_{L^2(\Omega)}
\le
\left(
\sum_{P\in\mathcal P}
J_{P,\max}
\left(
\mathcal D_x(g_P)
+
\kappa_x\mathcal D_y(g_P)
\right)^2
\right)^{1/2}.
\end{equation}

If several local approximations are available at the same physical point,
the final approximation is defined by a fixed patchwise assembly rule. Such a
rule specifies which local approximation contributes to the global output and
does not change the local form of the estimate.

\begin{remark}[Effect of boundary patch size]
\label{rem:jacobian-boundary-patch}
The Jacobian factor and the smoothness of the pulled-back function are the
two places where the geometry of a curved boundary patch enters the estimate.
For example, consider a top-type patch
\[
P_T=\{(x,y): a\le x\le b,\ y_B\le y\le \gamma(x)\},
\]
and set
\[
\Delta x=b-a,\qquad H(x)=\gamma(x)-y_B.
\]
The standard map
\[
x=a+\Delta x\,\xi,\qquad
y=y_B+\eta H(a+\Delta x\,\xi),
\qquad (\xi,\eta)\in[0,1]^2,
\]
has
\[
\det D\Psi=\Delta x\,H(a+\Delta x\,\xi).
\]
Thus the determinant is controlled by the local width and height of the
patch. More importantly, after normalizing the patch to a fixed reference
rectangle with a representative height \(H_0\), the curved boundary becomes
\[
\widehat\gamma(\xi)
=
\frac{\gamma(a+\Delta x\,\xi)-y_B}{H_0}.
\]
Its second derivative satisfies
\[
\widehat\gamma''(\xi)
=
\frac{(\Delta x)^2}{H_0}
\gamma''(a+\Delta x\,\xi).
\]
If the boundary patches have bounded aspect ratio, so that
\(H_0\asymp \Delta x\), then
\[
\|\widehat\gamma''\|_{L^\infty(0,1)}
\le
C\Delta x\,\|\gamma''\|_{L^\infty(a,b)}.
\]
Therefore, although the physical curvature of the boundary is unchanged, the
normalized curved side on each reference patch becomes closer to a straight
segment as the partition is refined. This reduces the non-affine geometric
distortion of the pulled-back function and helps decrease the error constants
associated with boundary patches.
\end{remark}

\section{Parameter calibration on rectangular patches}
\label{sec:parameter-calibration}

Before treating curved and general-domain patches, we first calibrate the
main local parameters of the tensor-product Fourier extension operator on
rectangular patches. Rectangular patches do not involve geometric distortion,
and therefore provide a clean setting for testing the approximation behavior
of the local algebraic procedure itself. In the \(x\)-direction the relevant
parameters are the extension parameter \(T_x\), the sampling ratio
\(\gamma_x\), and the local Fourier order \(N_x\), with analogous parameters
in the \(y\)-direction.

In the parameter tests, we use the separable oscillatory function
\[
    F(x,y)=\exp(\mathrm{i}\omega_x x)\exp(\mathrm{i}\omega_y y),
\]
where \(\omega_x\) and \(\omega_y\) denote the physical frequencies in the
\(x\)- and \(y\)-directions, respectively. Since the present calibration is
intended to identify the effective selection rule for the \(x\)-direction
parameters, the figures mainly report the dependence of the error on
\(T_x\), \(N_x\), \(\gamma_x\), and \(\omega_x\). We also tested the influence
of the \(y\)-direction parameters on the \(x\)-direction calibration, but the
observed effect was minor and did not change the admissible ranges or the
threshold behavior reported below. For this reason, only the representative
\(x\)-direction tests are displayed.

\begin{figure}[htbp]
\centering
\begin{subfigure}[b]{0.32\textwidth}
    \centering
    \includegraphics[width=\textwidth]{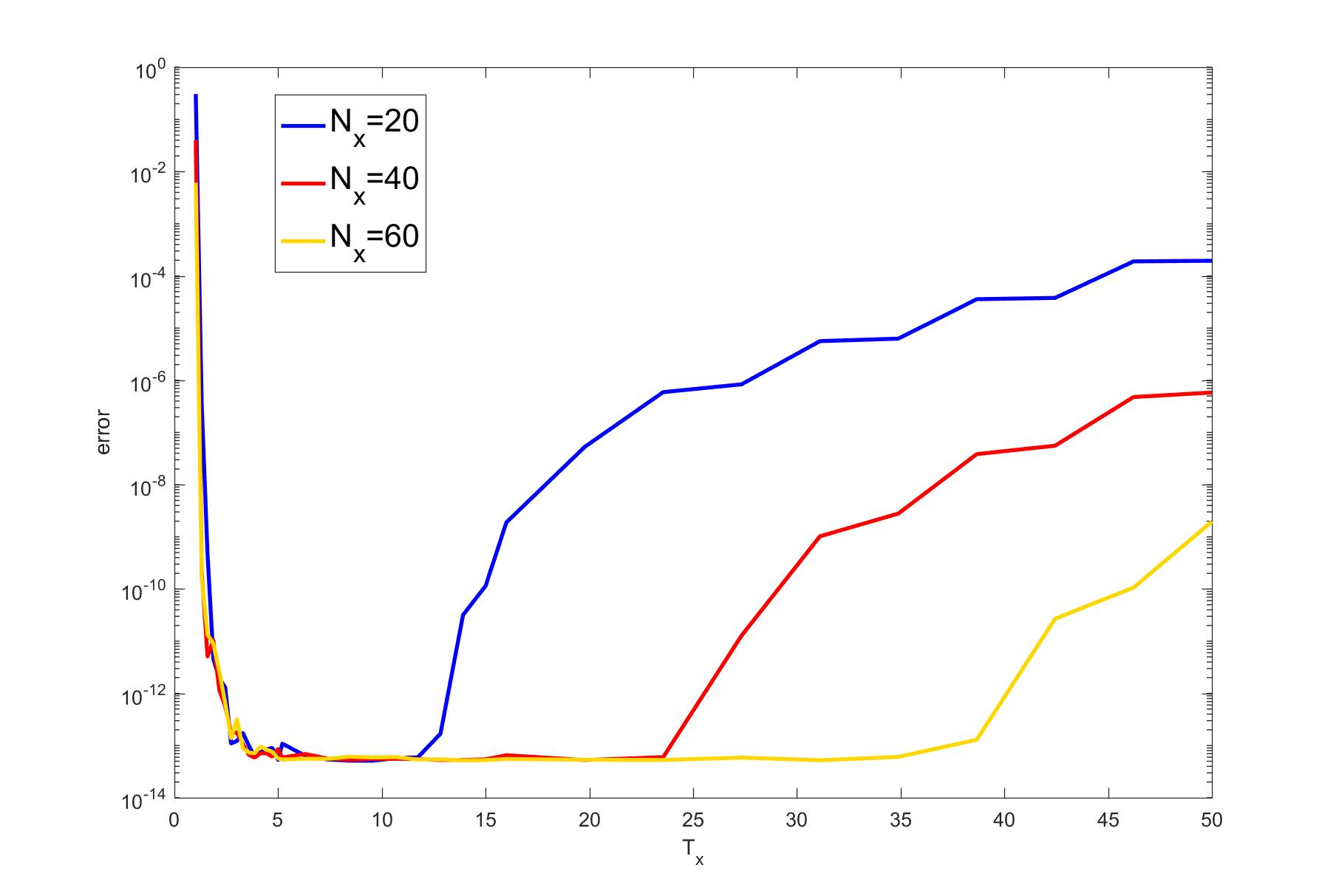}
    \caption{Variation of \(N_x\).}
\end{subfigure}
\hfill
\begin{subfigure}[b]{0.32\textwidth}
    \centering
    \includegraphics[width=\textwidth]{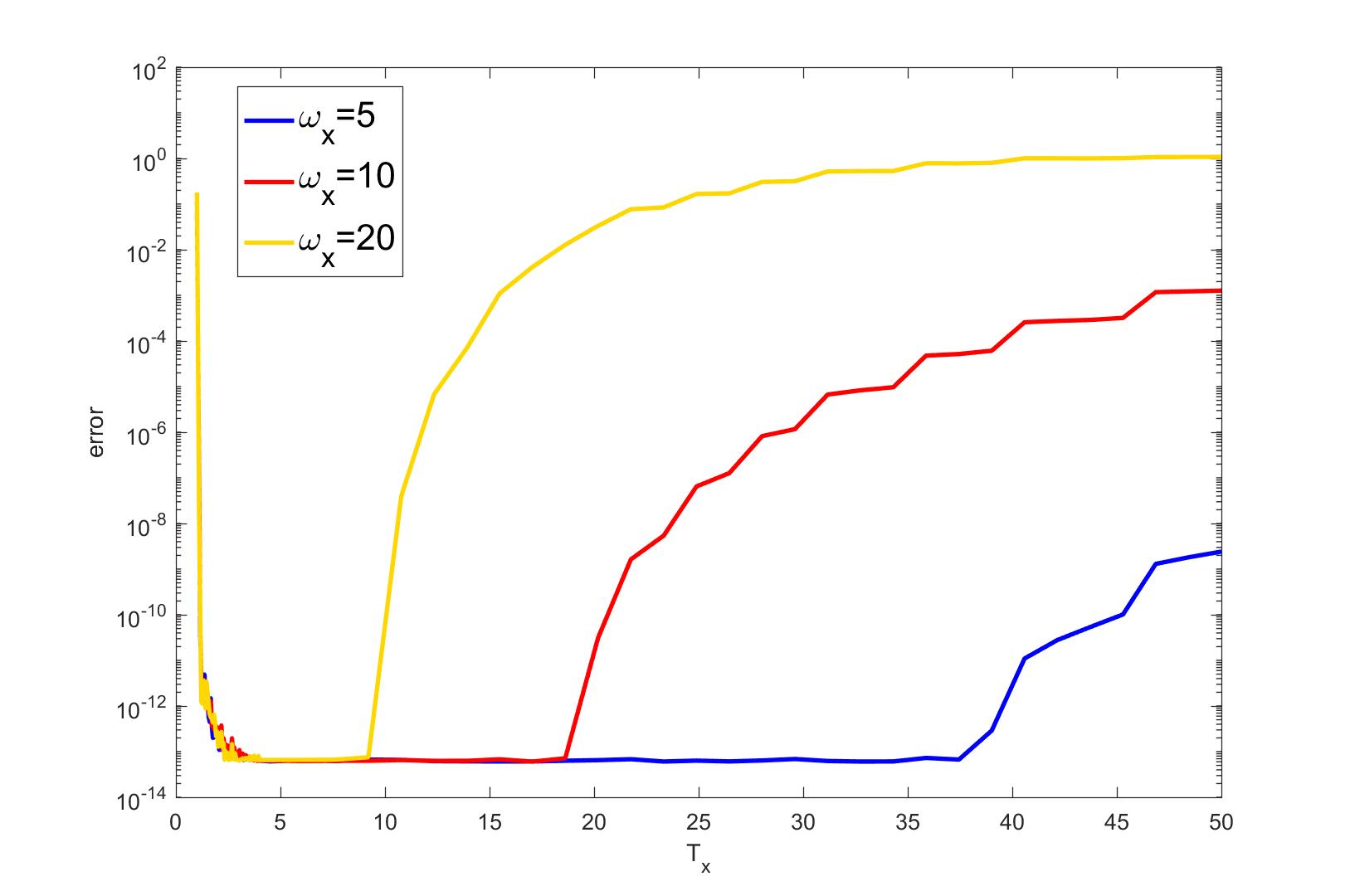}
    \caption{Variation of \(\omega_x\).}
\end{subfigure}
\hfill
\begin{subfigure}[b]{0.32\textwidth}
    \centering
    \includegraphics[width=\textwidth]{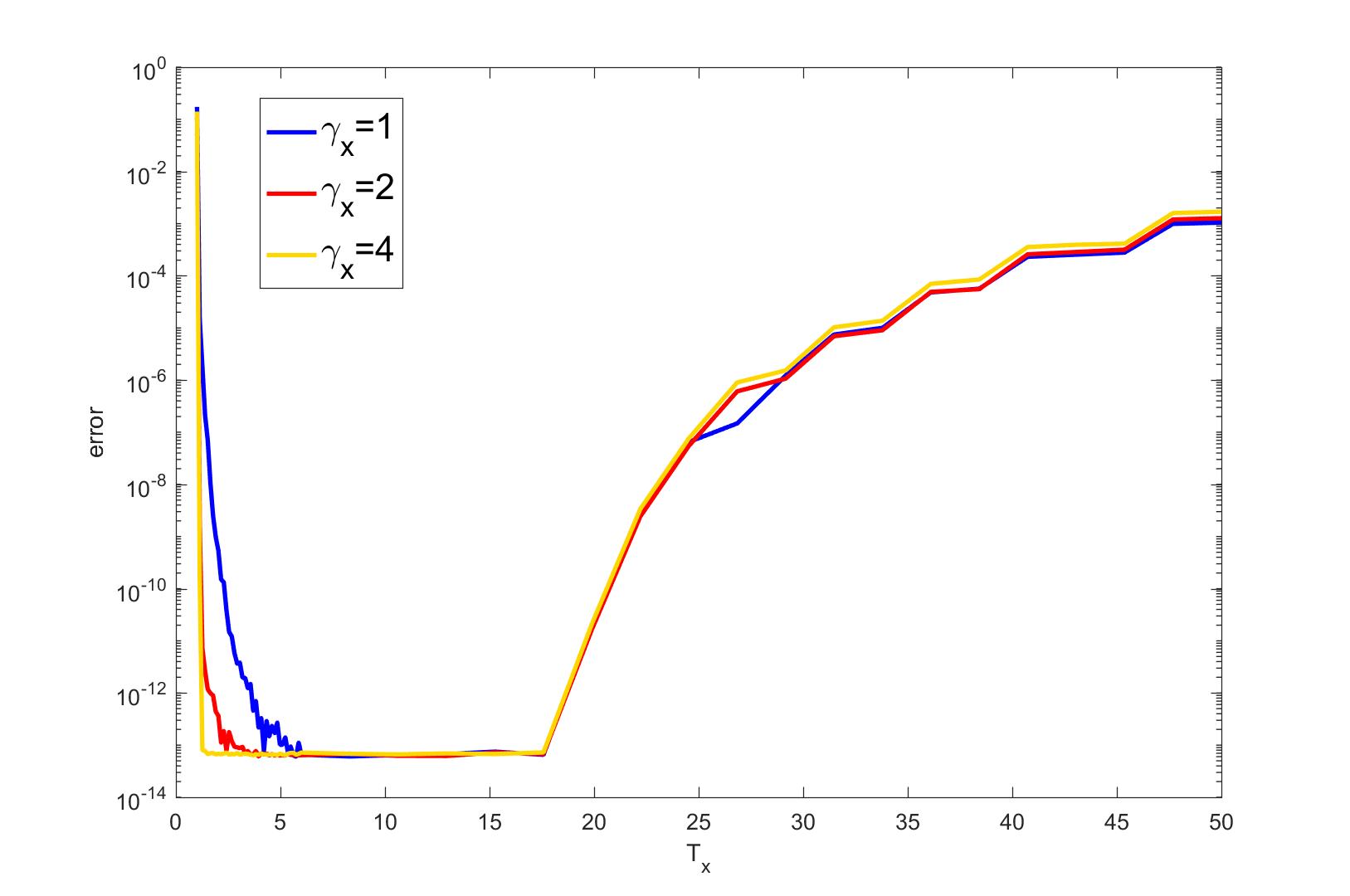}
    \caption{Variation of \(\gamma_x\).}
\end{subfigure}

\vspace{2mm}

\begin{subfigure}[b]{0.32\textwidth}
    \centering
    \includegraphics[width=\textwidth]{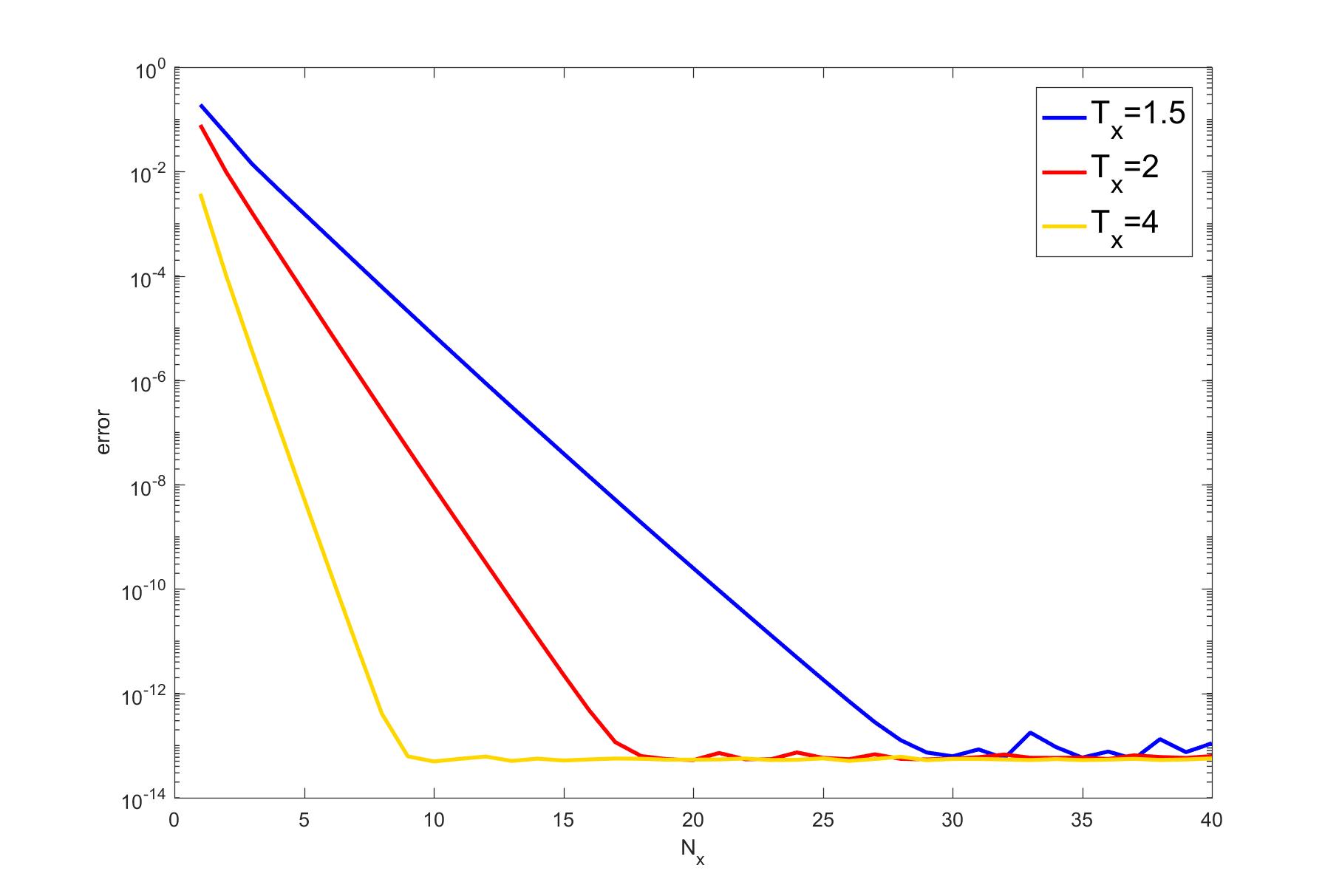}
    \caption{Variation of \(T_x\).}
\end{subfigure}
\hfill
\begin{subfigure}[b]{0.32\textwidth}
    \centering
    \includegraphics[width=\textwidth]{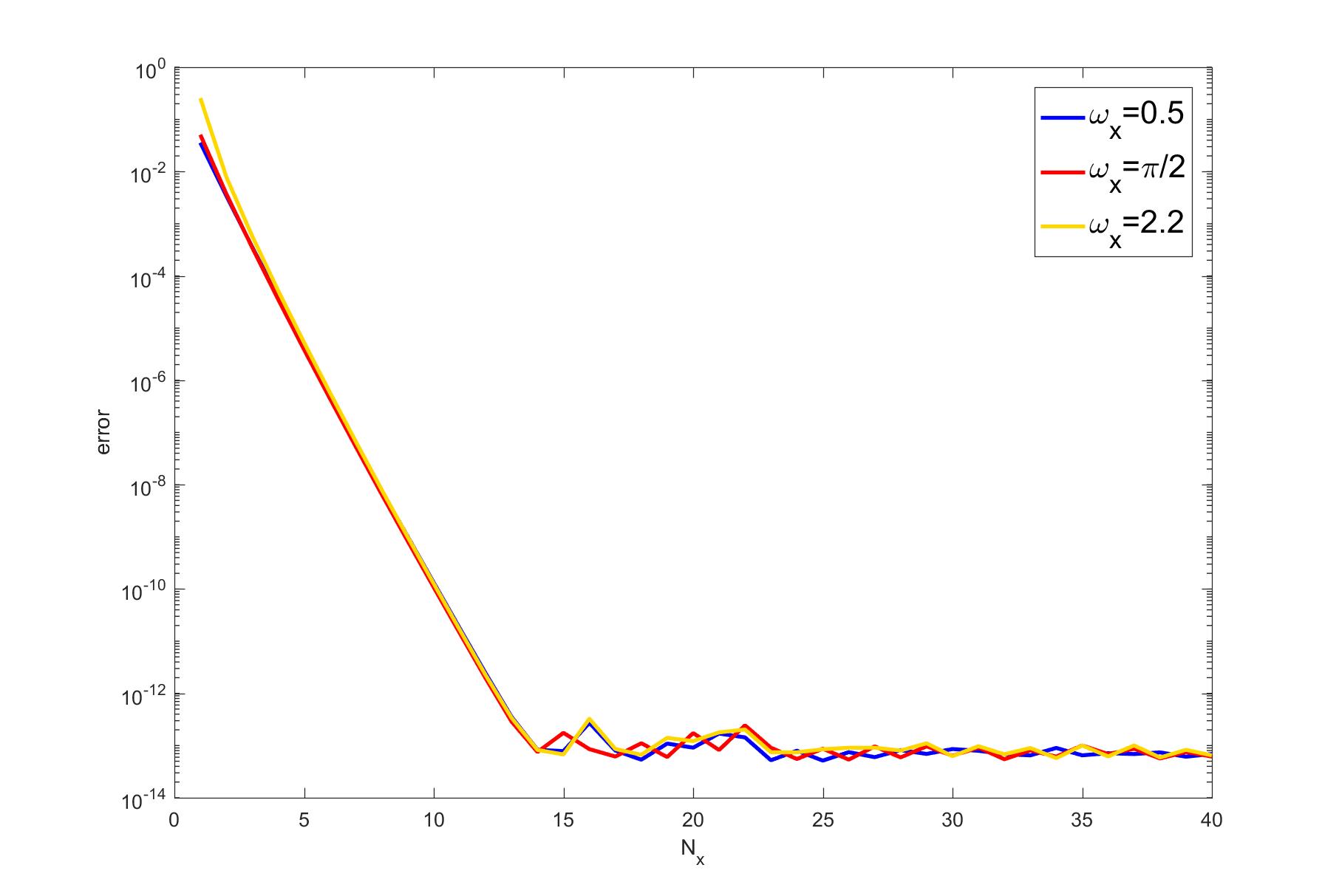}
    \caption{Variation of \(\omega_x\).}
\end{subfigure}
\hfill
\begin{subfigure}[b]{0.32\textwidth}
    \centering
    \includegraphics[width=\textwidth]{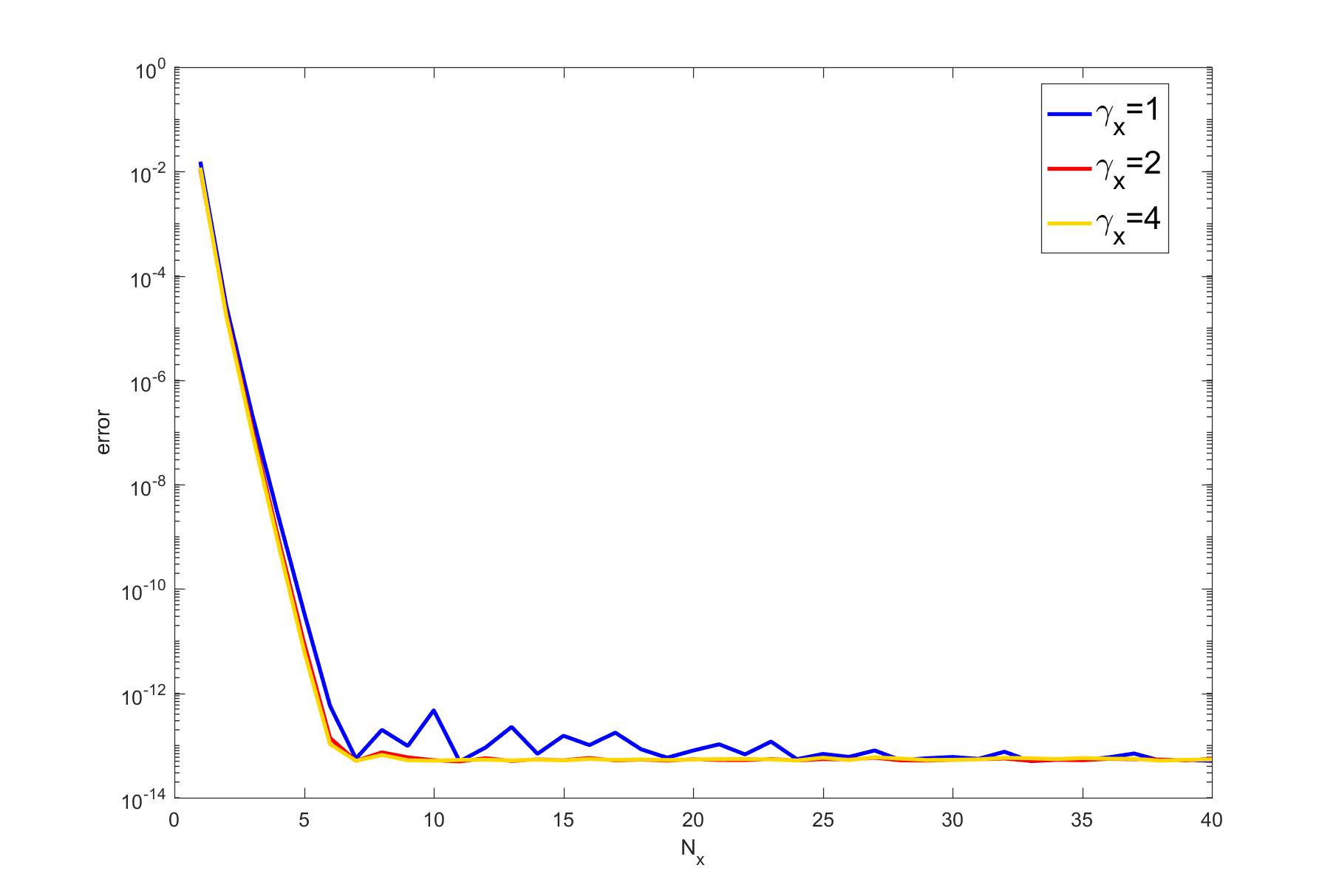}
    \caption{Variation of \(\gamma_x\).}
\end{subfigure}
\caption{Representative parameter tests on rectangular patches for the
separable oscillatory function
\(F(x,y)=\exp(\mathrm{i}\omega_x x)\exp(\mathrm{i}\omega_y y)\).
The first row shows the dependence of the error on the extension parameter
\(T_x\), and the second row shows the dependence on the local Fourier order
\(N_x\).}
\label{fig:parameter-test}
\end{figure}

The numerical tests show that the tensor-product construction inherits the
essential parameter-selection mechanism of the one-dimensional local Fourier
extension method. For fixed local frequency and Fourier order, there is an
admissible range
\[
    T_{x,\min}\le T_x\le T_{x,\max}
\]
in which the approximation reaches the accuracy level determined by the TSVD
threshold and double precision. The lower bound \(T_{x,\min}\) is mainly
controlled by the sampling ratio \(\gamma_x\), while the upper admissible
range is affected by the local oscillation level \(\omega_x\) and the Fourier
order \(N_x\). Once \(T_x\) lies in the admissible range, the error decreases
rapidly as \(N_x\) passes a threshold value. The additional cross tests with
respect to the \(y\)-direction parameters indicate that this \(x\)-direction
selection rule is only weakly affected by variations in the transverse
direction.

Figure~\ref{fig:parameter-test} summarizes representative tests. The first
row shows how the error depends on the extension parameter \(T_x\), while the
second row shows its dependence on the Fourier order \(N_x\). The results
confirm that \(T_x\), \(\gamma_x\), and \(N_x\) should be selected jointly.

To quantify the effect of oversampling, Table~\ref{tab:T1} lists the
observed lower admissible bounds \(T_{x,\min}\) for several choices of
\(\gamma_x\). As in the one-dimensional local Fourier extension method,
increasing \(\gamma_x\) decreases the lower threshold of the admissible range.
In particular, the choice \(T=4\) is not admissible for \(\gamma=1\), but
becomes admissible when the sampling ratio is mildly increased to
\(\gamma=1.2\).

\begin{table}[htbp]
\begin{center}
\caption{Observed lower admissible bound \(T_{x,\min}\) for several
oversampling ratios \(\gamma_x\) on rectangular patches.}
\label{tab:T1}
\small
{\begin{tabular*}{0.8\textwidth}{@{\extracolsep\fill}cccccc}
\toprule
& \(\gamma=1\) & \(\gamma=1.2\) & \(\gamma=1.5\) & \(\gamma=2\) & \(\gamma=4\) \\
\midrule
\(T_{x,\min}\) & 5.5 & 3.9 & 2.9 & 2.2 & 1.2 \\
\bottomrule
\end{tabular*}}
\end{center}
\end{table}

\begin{table}[htbp]
\begin{center}
\caption{Observed threshold values of the local Fourier order \(N_x\)
required to reach the target accuracy for different extension parameters
\(T_x\).}
\label{tab:T2}
\small
{\begin{tabular*}{0.8\textwidth}{@{\extracolsep\fill}cccccccccc}
\toprule
& \(T_x=1.2\) & \(T_x=1.5\) & \(T_x=2\) & \(T_x=3\) & \(T_x=4\) & \(T_x=6\) \\
\midrule
\(N_x\) & 58 & 27 & 17 & 13 & 10 & 9 \\
\bottomrule
\end{tabular*}}
\end{center}
\end{table}

The second relation concerns the local Fourier order \(N_x\). Once \(T_x\)
is admissible, the numerical error drops rapidly when \(N_x\) exceeds a
threshold value. Table~\ref{tab:T2} records the observed thresholds for
several choices of \(T_x\). Very small extension parameters require much
larger Fourier orders: for example, \(T_x=1.2\) requires \(N_x=58\), whereas
\(T_x=4\) requires only \(N_x=10\), and \(T_x=6\) requires \(N_x=9\).
Thus increasing \(T_x\) can reduce the required Fourier order, but this
benefit must be balanced against local resolution.

The reason for not simply taking \(T_x\) as large as possible is local
resolution. Suppose that the target function contains an \(x\)-direction
oscillatory component with physical frequency \(\omega_x\). On a patch of
width \(\Delta x\), after rescaling to the Fourier extension interval
\([0,2\pi/T_x]\), the effective local frequency is proportional to
\[
    \mu_x=\frac{\omega_x\Delta x\,T_x}{2\pi}.
\]
For a local Fourier space with modes \(|\ell|\le N_x\), a practical
resolution condition is
\[
    \mu_x\lesssim N_x.
\]
Thus, for fixed \(N_x\) and patch size, a smaller admissible value of \(T_x\)
improves the ability of the local approximation to resolve oscillations.
Equivalently, for a fixed oscillatory function, reducing \(T_x\) can reduce
the required number of subintervals. Therefore, although \(T=6\) slightly
reduces the Fourier order required to reach the target accuracy compared with
\(T=4\), it also weakens the local resolution capability.

The parameter choice should also account for the cost of the local
operations. In one coordinate direction, the number of local sampling points
satisfies
\[
    m_x\approx \gamma_x(2N_x+1).
\]
Larger \(\gamma_x\) or \(N_x\) increases the size of the local matrices and
the number of retained singular components in the TSVD procedure. Since the
reference matrices are precomputed and reused, this cost remains local; but
an efficient parameter choice should keep the local sampling size and the
number of retained modes moderate. The global complexity of the full
patchwise method will be discussed in Section~\ref{sec:patch-construction}.

The robust one-dimensional parameter set
\[
    T=6,\qquad \gamma=1,\qquad N=9
\]
is stable and effective. The present rectangular-patch tests indicate that a
slightly smaller extension parameter, together with mild oversampling,
provides a better balance for the two-dimensional patchwise computations.
Specifically, reducing \(T\) from \(6\) to \(4\) improves local resolution,
while Table~\ref{tab:T1} shows that increasing the sampling ratio from
\(\gamma=1\) to \(\gamma=1.2\) is sufficient to keep \(T=4\) inside the
admissible range. Table~\ref{tab:T2} further shows that the corresponding
Fourier order \(N=10\) is sufficient to reach the target accuracy and remains
close to the value \(N=9\) used with \(T=6\).

Based on these observations, we adopt
\[
    T_x=T_y=4,\qquad
    \gamma_x=\gamma_y=1.2,\qquad
    N_x=N_y=10
\]
as the default local parameter set in the numerical experiments. This choice
keeps the extension parameter small enough to retain local resolution, while
maintaining stability through mild oversampling. The local Fourier order is
moderate, so the corresponding local matrices remain small and can be reused
efficiently.

These values should be understood as practical reference choices for
double-precision computations with a TSVD threshold close to machine
precision. If higher-precision arithmetic or a different truncation threshold
is used, the admissible range of \(T\), the number of retained singular
components, and the recommended values of \(N\) and \(\gamma\) should be
recalibrated.

\section{Approximation on curved trapezoidal patches}
\label{sec:curved-trapezoid}

In this section, we describe the treatment of one-side curved trapezoidal
patches. These patches are the basic boundary components used later in the
scan-based construction of general domains. We first present the standard
LFE-based transfer for a smooth curved side, then show why local subdivision
improves the accuracy, and finally introduce a smooth-cover correction for
boundary patches with small-scale rough perturbations.

\subsection{Sampling nodes and LFE-based vertical transfer}
\label{subsec:curved-vertical-transfer}

It is sufficient to describe the procedure for a representative top-type
curved trapezoidal patch
\[
    \Omega_b=\{(x,y):x\in[a,b],\ 0\le y\le b(x)\},
\]
where \(b(x)>0\) describes the curved upper boundary. Bottom, left, and right
patches are treated analogously, by reversing the normal direction or
exchanging the roles of \(x\) and \(y\).

For the top patch, we use an \(x\)-aligned semi-structured sampling strategy.
The nodes in the \(x\)-direction are equispaced,
\[
    x_i=a+i\Delta x,\qquad i=0,1,\ldots,m_x-1.
\]
For each fixed \(x_i\), the physical sampling nodes lie on the vertical
segment
\[
    \{(x_i,y):0\le y\le b(x_i)\}.
\]
In the implementation, the boundary point \(y=b(x_i)\) is included so that
the curved side is explicitly incorporated into the sampled data.  Figure~\ref{fig:curved_nodes} illustrates this sampling pattern. The blue
circles denote the physical sampling nodes on the curved trapezoidal patch,
while the red dots indicate the mapped target nodes after the vertical
transfer. For graphical clarity, only a reduced number of nodes is displayed;
the actual computations use the same sampling strategy with the prescribed
local resolution.

\begin{figure}[htbp]
\centering
\includegraphics[width=0.5\textwidth]{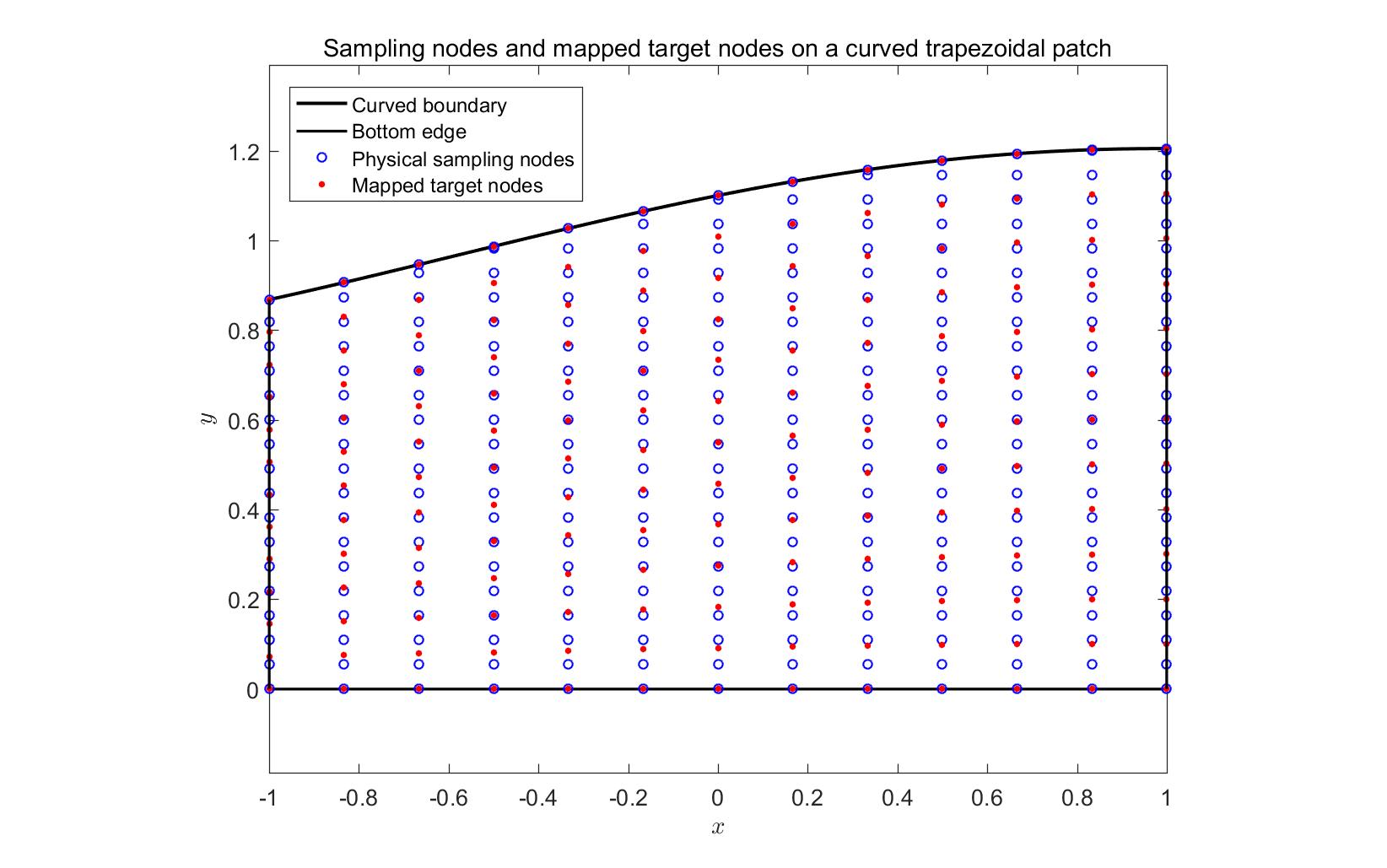}
\caption{Sampling and vertical-transfer nodes on a representative top-type
curved trapezoidal patch.}
\label{fig:curved_nodes}
\end{figure}

To convert the data on the curved patch into a rectangular tensor-product
array, we introduce target nodes on each vertical line,
\[
    y_{i,j}^{*}
    =
    \frac{j}{m_y-1}\,b(x_i),
    \qquad
    j=0,1,\ldots,m_y-1.
\]
These nodes correspond to an equispaced discretization in the normalized
vertical coordinate
\[
    \eta=\frac{y}{b(x)}\in[0,1].
\]
The values at the target nodes are obtained by applying the one-dimensional
local Fourier extension procedure along each vertical line. In this way, the
data on the curved trapezoid are transferred to a fixed-size rectangular
array without using low-order interpolation.

After the vertical LFE transfer, we obtain
\[
    F_{j,i}
    \approx
    u\!\left(x_i,\frac{j}{m_y-1}b(x_i)\right),
    \qquad
    i=0,\ldots,m_x-1,\quad j=0,\ldots,m_y-1.
\]
This array is then processed by the tensor-product local Fourier extension
procedure for rectangular patches. Thus, the standard treatment of a curved
trapezoidal patch consists of two local stages: a one-dimensional LFE-based
line transfer followed by a tensor-product LFE approximation on the resulting
rectangular data.

\subsection{Effect of vertical subdivision on curved trapezoidal patches}
\label{subsec:curved-trapezoid-splitting}

We next examine the influence of the curved boundary geometry on the
approximation accuracy. The purpose is to show that, when the curved side has
a relatively large variation on a single patch, further local subdivision can
reduce the geometric distortion and improve the approximation.

We consider
\[
    b(x)=1+0.5x^2+0.2\cos(\pi x+0.30),
    \qquad -1\le x\le 1,
\]
and the corresponding patch
\[
    \Omega_b
    =
    \{(x,y):-1\le x\le 1,\ 0\le y\le b(x)\}.
\]
Although \(b(x)\) is smooth, its height variation is sufficient to reveal the
effect of using a single large curved patch. We compare the following two
treatments:
\[
\begin{array}{ll}
\text{one-patch treatment:}
&
\Omega_b \text{ is approximated as one curved trapezoidal patch},\\[1mm]
\text{two-patch treatment:}
&
\Omega_b \text{ is divided vertically at } x=0
\text{ into two curved trapezoidal patches}.
\end{array}
\]
The two test functions are
\[
    u_1(x,y)=\frac{\sin(xy/2)}{1+y^2/4},
    \qquad
    u_2(x,y)=\cos(1.45x)\sin(1.37y).
\]

\begin{figure}[htbp]
\centering
\begin{subfigure}[b]{0.49\textwidth}
    \centering
    \includegraphics[width=\textwidth]{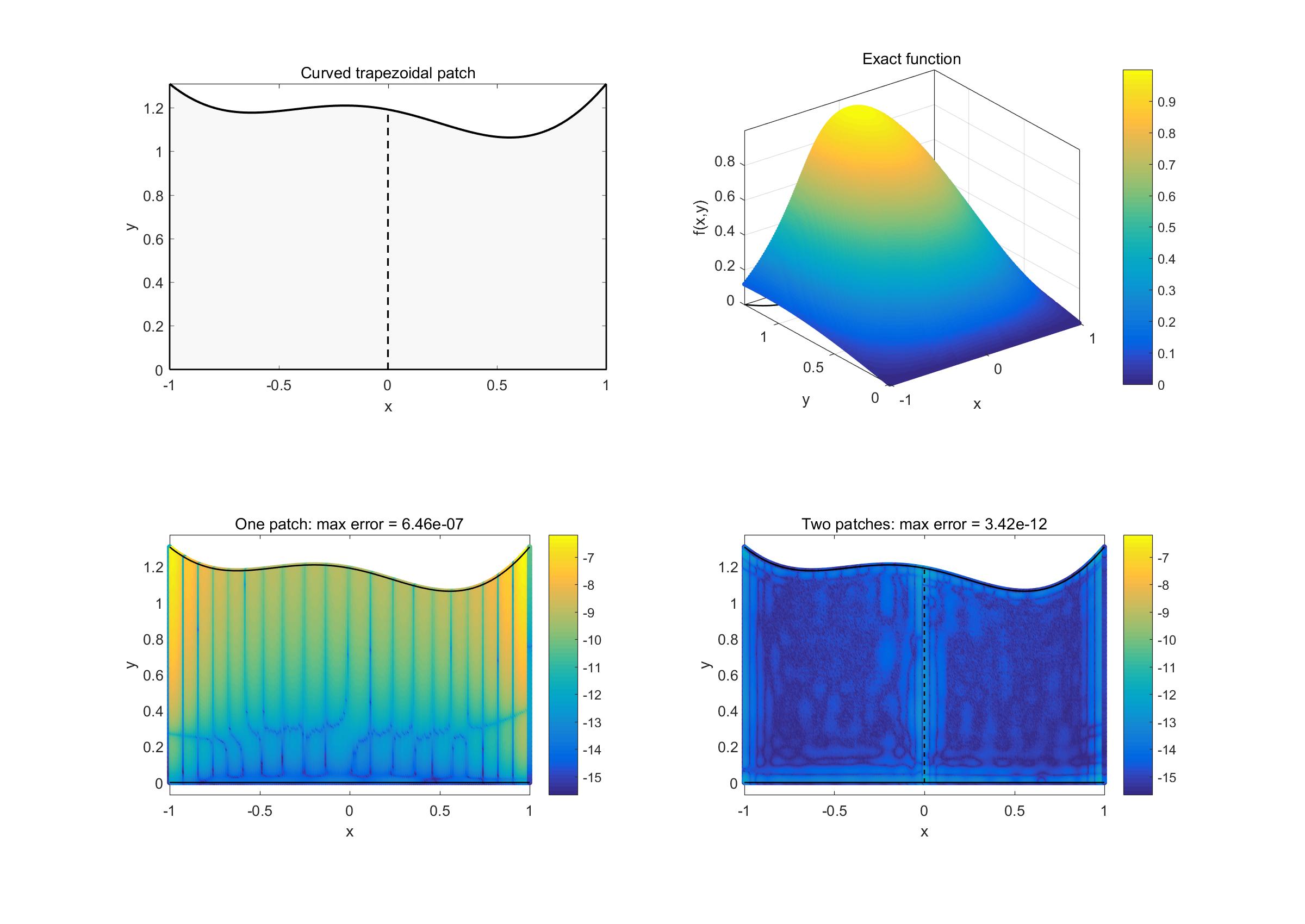}
    \caption{\(u_1(x,y)=\sin(xy/2)/(1+y^2/4)\). The maximum error is reduced
    from \(6.46\times10^{-7}\) to \(3.42\times10^{-12}\).}
    \label{fig:curved_split_u1}
\end{subfigure}
\hfill
\begin{subfigure}[b]{0.49\textwidth}
    \centering
    \includegraphics[width=\textwidth]{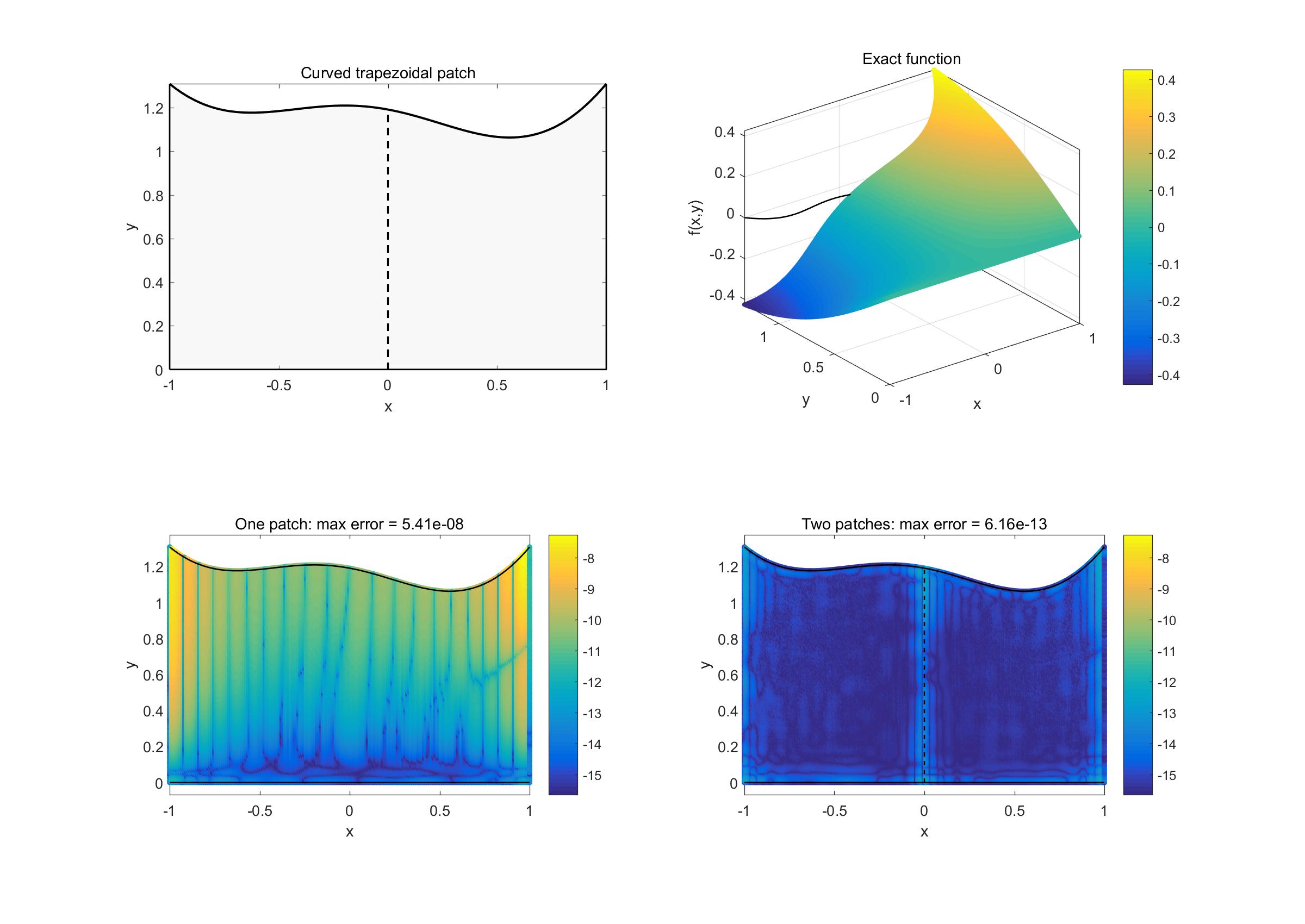}
    \caption{\(u_2(x,y)=\cos(1.45x)\sin(1.37y)\). The maximum error is reduced
    from \(5.41\times10^{-8}\) to \(6.16\times10^{-13}\).}
    \label{fig:curved_split_u2}
\end{subfigure}
\caption{Effect of vertical subdivision on a curved trapezoidal patch. Splitting the patch reduces the geometric distortion and significantly improves the approximation accuracy.}
\label{fig:curved_split}
\end{figure}

Figure~\ref{fig:curved_split} shows the results. For \(u_1\), the one-patch
treatment gives a maximum error of about \(6.46\times10^{-7}\), whereas the
two-patch treatment reduces the maximum error to about
\(3.42\times10^{-12}\). For \(u_2\), the maximum error is reduced from about
\(5.41\times10^{-8}\) to about \(6.16\times10^{-13}\). Thus, in both cases,
splitting the curved trapezoidal patch into two smaller patches leads to a
substantial improvement, with the post-splitting errors approaching the level
of machine precision.

This observation is consistent with the geometric discussion in
Section~\ref{sec:error-estimate}. A curved boundary patch is mapped to a
reference rectangle before the tensor-product LFE procedure is applied. If a
patch contains a relatively long curved side, the normalized boundary may
have a larger non-affine component, which can amplify the local
approximation error. After subdivision, each subpatch contains a shorter
portion of the curved boundary; the normalized curve is closer to a straight
segment, and the geometric contribution to the error constant is reduced.

\subsection{Smooth-cover correction for mildly rough boundary patches}
\label{subsec:smooth-cover-correction}

The preceding construction assumes that the curved side can be used directly
as a local boundary graph. We now describe a local correction for boundary
patches whose curved side contains small-scale rough perturbations. The
correction does not change the scan-based patch classification; it only
modifies the solver used for the affected boundary patches.

For clarity, we describe the procedure for a top-type patch
\[
    P=\{(x,y):x_L\le x\le x_R,\ y_B\le y\le b(x)\},
\]
where \(b(x)\) is the original rough boundary segment. Bottom patches are
handled by reversing the normal direction, and left or right patches are
handled by exchanging the roles of \(x\) and \(y\).

Since the boundary has already been localized by the patch partition, we
assume that the height variation of \(b(x)\) on \([x_L,x_R]\) is moderate.
Let \(p_d(x)\) be a low-degree polynomial fit, with \(d=3\) in the numerical
experiments, to boundary samples on this interval. We define a smooth cover
\[
    b_c(x)=p_d(x)+\delta,
    \qquad
    \delta=\max_i\{b(x_i)-p_d(x_i)\}+\delta_0,
\]
where \(\{x_i\}\) are sufficiently dense boundary samples and \(\delta_0>0\)
is a small safety margin. Then \(b_c(x)\) stays above the sampled rough
boundary. The computation is performed on the smooth covering patch
\[
    P_c=\{(x,y):x_L\le x\le x_R,\ y_B\le y\le b_c(x)\},
\]
but the final output is retained only in the original physical patch \(P\).

The values on the artificial boundary \(y=b_c(x)\) are not assumed to be
known. They are recovered column by column from the available data inside
\(P\). For a fixed vertical line \(x=x_k\), let
\[
    y_{k,1},\ldots,y_{k,m_k}
\]
be the available nodes in the original region, consisting of the uniform
background-grid nodes below the rough boundary together with the boundary
intersection \(b(x_k)\). We add one artificial node at
\[
    y_{k,m_k+1}=b_c(x_k).
\]
The corresponding data vector is
\[
    {\bf f}_k=
    \bigl(f(x_k,y_{k,1}),\ldots,f(x_k,y_{k,m_k}),\alpha_k\bigr)^T,
\]
where \(\alpha_k\approx f(x_k,b_c(x_k))\) is the only unknown.

Since the nodes \(y_{k,1},\ldots,y_{k,m_k},b_c(x_k)\) are generally
nonuniform, a one-dimensional Fourier extension matrix \(A_k\) is constructed
separately for this column. Let
\[
    A_k=U_k\Sigma_k V_k^*
\]
be its singular value decomposition, and let \(U_{k,0}\) contain the left
singular vectors corresponding to singular values below the truncation
threshold. A stable Fourier extension data vector should have a small
component in these discarded left-singular directions. Therefore the
artificial value is determined from the consistency condition
\[
    U_{k,0}^*{\bf f}_k\approx 0 .
\]
Writing
\[
    U_{k,0}
    =
    \begin{pmatrix}
    U_{k,0}^{\rm in}\\
    u_{k,0}^{\rm out}
    \end{pmatrix},
\]
where the last row corresponds to the artificial cover node, we compute
\[
    \alpha_k
    =
    \arg\min_{\alpha}
    \left\|
    (U_{k,0}^{\rm in})^*{\bf f}_{k}^{\rm in}
    +
    \overline{u_{k,0}^{\rm out}}\,\alpha
    \right\|_2 .
\]
Only one unknown is recovered on each sampling line. After this completion,
the column data are projected onto the stable Fourier extension space and
evaluated on the usual \(m\)-point target grid associated with \(P_c\). The
completed data are then processed by the standard top-type curved-patch solver
from Section~\ref{subsec:curved-vertical-transfer}, and the final values used
in the global approximation are restricted to \(P\). Thus the rough boundary
is used to determine the physical output region, while the smooth cover is
used only to provide a stable local computational patch.

\begin{figure}[H]
\centering
\includegraphics[width=0.95\textwidth]{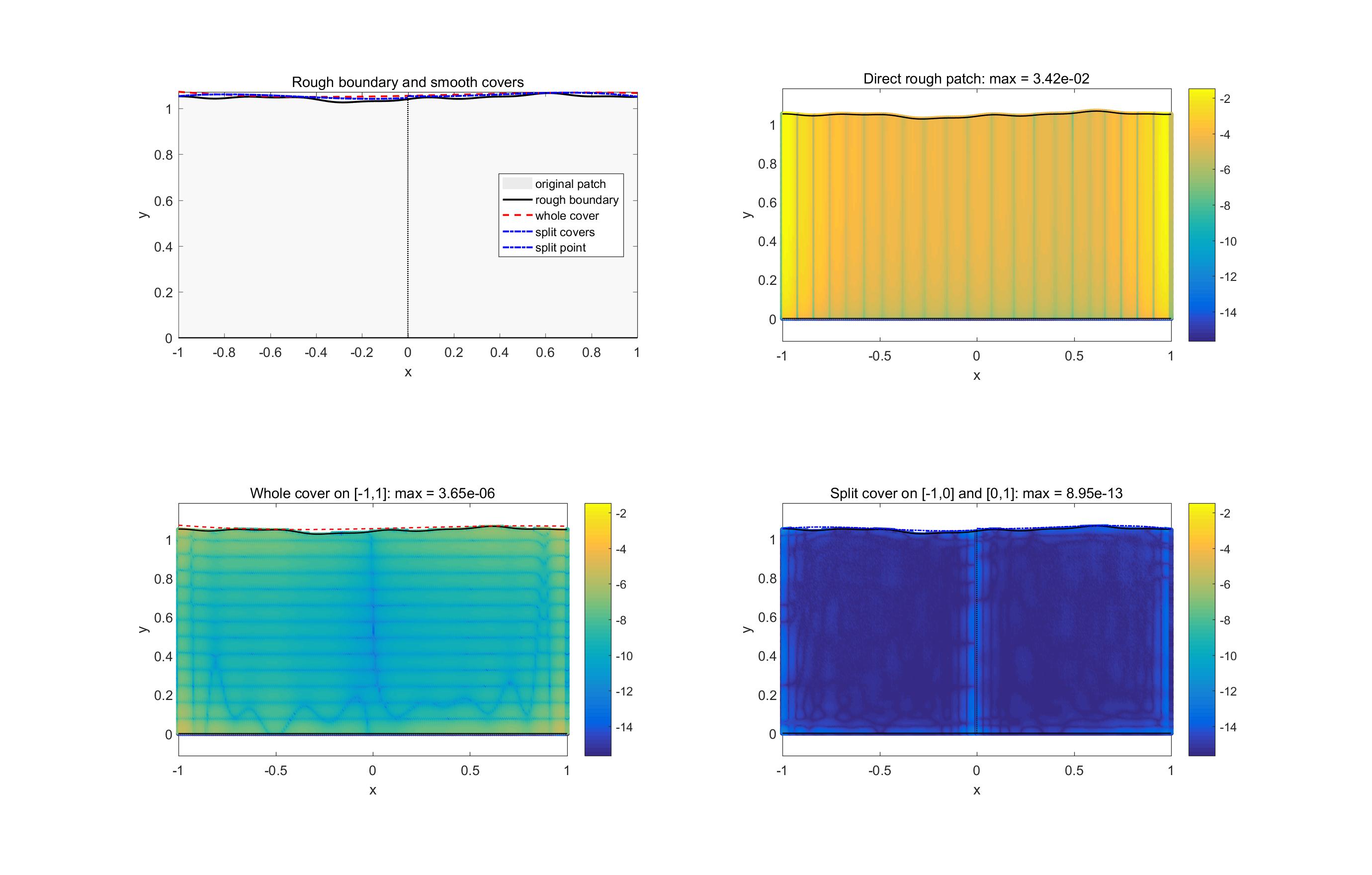}
\caption{Patch-level test of the smooth-cover correction on a mildly rough
top boundary. The three error plots correspond to direct treatment, one global
smooth cover, and two local smooth covers after subdivision.}
\label{fig:nonsmooth-cover}
\end{figure}

To verify the local correction mechanism, we consider a mildly rough top-type
patch
\[
    P=\{(x,y):-1\le x\le 1,\ 0\le y\le b(x)\}.
\]
Figure~\ref{fig:nonsmooth-cover} compares three treatments: direct use of the
rough boundary, one smooth cover on the whole interval \([-1,1]\), and two
local smooth covers recomputed on \([-1,0]\) and \([0,1]\). The maximum error
is reduced from \(3.42\times10^{-2}\) to \(3.65\times10^{-6}\), and further
to \(8.95\times10^{-13}\) after local subdivision.

This patch-level test is consistent with the subdivision effect observed in
Section~\ref{subsec:curved-trapezoid-splitting}. In both the smooth and mildly
rough cases, local subdivision reduces the difficulty of the boundary patch:
it decreases the geometric distortion for smooth curved sides and produces
closer smooth covers for rough boundary portions. This provides an additional
verification of the local patchwise principle underlying the proposed method.

\section{Patch construction for general domains and complexity analysis}
\label{sec:patch-construction}

\subsection{Scan-based partition into local patches}
\label{subsec:scan-partition}

We now describe how a general two-dimensional domain is decomposed into local
patches suitable for the tensor-product Fourier extension procedure. The
basic idea is to embed the domain into a Cartesian background grid, scan the
boundary curve once, and use the scanning result to construct rectangular
interior patches and one-side curved boundary patches.

Let \(\Omega\subset\mathbb{R}^2\) be a bounded domain enclosed by a closed
curve
\[
    \Gamma=\{(x(t),y(t)):0\le t<2\pi\}.
\]
Assume that
\[
    \Omega\subset [a,b]\times[c,d],
\]
and introduce a Cartesian partition
\[
    a=x_0<x_1<\cdots<x_{K_x}=b,\qquad
    c=y_0<y_1<\cdots<y_{K_y}=d.
\]
The corresponding background cells are denoted by
\[
    Q_{ij}=[x_{i-1},x_i]\times[y_{j-1},y_j],
    \qquad
    1\le i\le K_x,\quad 1\le j\le K_y .
\]

The boundary is scanned in a fixed orientation after computing its
intersections with the Cartesian grid lines. During this scan, each boundary
arc is assigned one of four local directional types
\[
    L,\qquad T,\qquad R,\qquad B,
\]
according to whether the boundary portion is locally represented as the left,
top, right, or bottom side of a computational patch. This classification is
local: the same directional type may appear in several separated portions of
a general boundary, and transitions between different types may occur more
than once.

Interior cells are treated as rectangular patches and are labeled as
\(\texttt{rect}\). Exterior regions are discarded. Boundary arcs generate
one-side curved trapezoidal patches. The four typical forms are
\[
\begin{aligned}
    \Omega_L &= \{(x,y): y_B\le y\le y_T,\ \gamma_L(y)\le x\le x_R\},\\
    \Omega_R &= \{(x,y): y_B\le y\le y_T,\ x_L\le x\le \gamma_R(y)\},\\
    \Omega_T &= \{(x,y): x_L\le x\le x_R,\ y_B\le y\le \gamma_T(x)\},\\
    \Omega_B &= \{(x,y): x_L\le x\le x_R,\ \gamma_B(x)\le y\le y_T\}.
\end{aligned}
\]
Here \(\gamma_L,\gamma_R,\gamma_T,\gamma_B\) are local representations of the
boundary curve. Each boundary patch has one curved side inherited from
\(\Gamma\) and three straight sides aligned with the Cartesian grid. This
form is convenient for the subsequent computation, since the patch can be
sampled along vertical or horizontal lines and converted to a tensor-product
data array.

Near transition portions of the boundary, additional covering patches are
introduced. These corner covers are not new approximation types; each of them
is converted into one of the four standard types \(L,T,R\), or \(B\),
according to the local sampling direction. Thus all boundary regions are
eventually handled by the same four curved-patch solvers.

Figure~\ref{fig:scan-to-block} illustrates the scan-based construction for a
smooth curved boundary. The left panel shows the boundary scan and the local
directional labels. The right panel shows the corresponding patch
classification: interior rectangles are labeled as \(\texttt{rect}\),
exterior regions are shown as the gray background, and boundary patches are
displayed as one-side curved trapezoidal patches. The red outlines indicate
the additional cover patches used near transition regions.

\begin{figure}[htbp]
\centering
\begin{subfigure}[b]{0.45\textwidth}
\centering
\includegraphics[width=\textwidth]{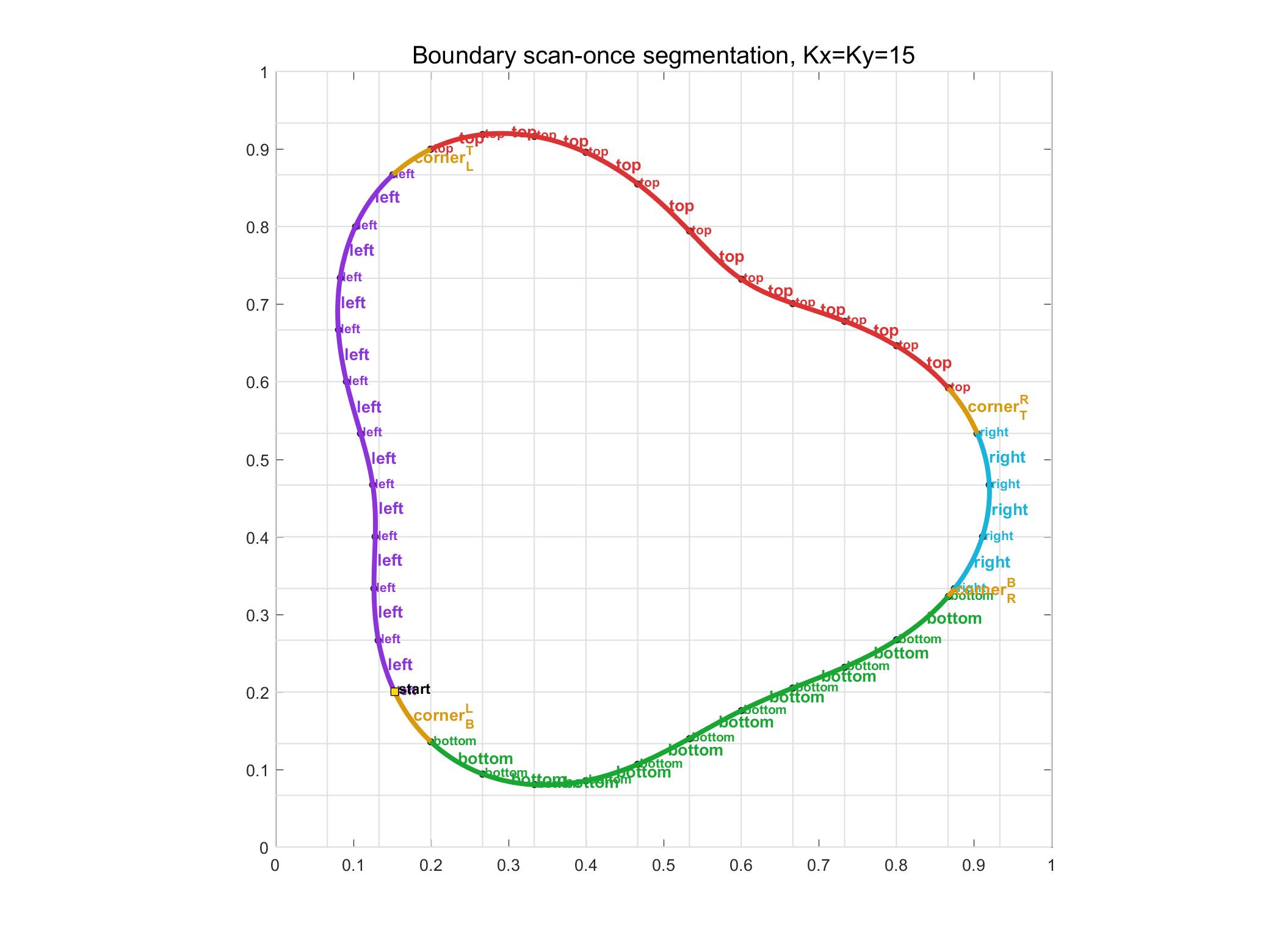}
\caption{Boundary scan and local directional labeling.}
\label{fig:boundary-scan-sub}
\end{subfigure}
\hfill
\begin{subfigure}[b]{0.45\textwidth}
\centering
\includegraphics[width=\textwidth]{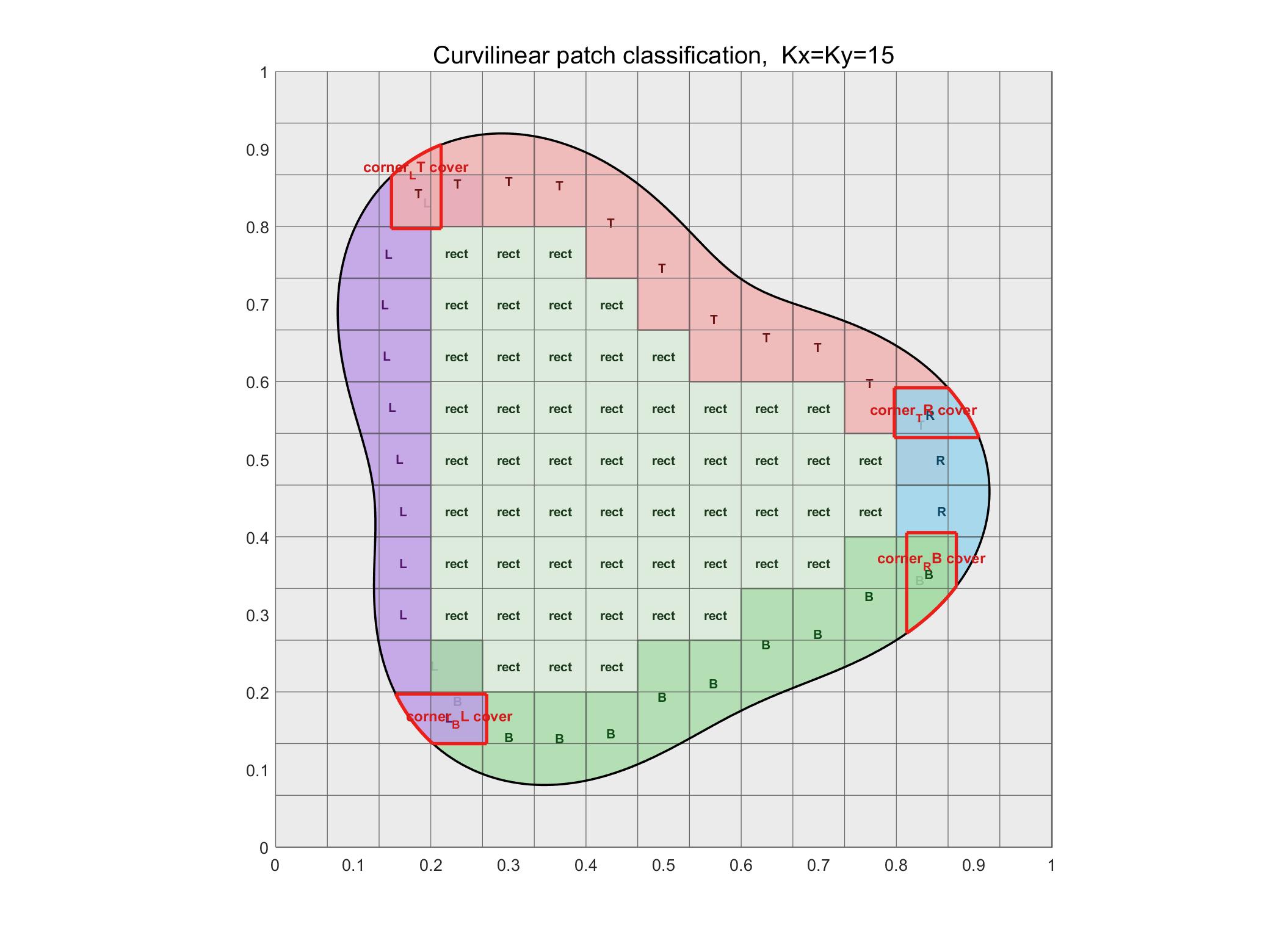}
\caption{Curvilinear patch classification and corner covers.}
\label{fig:block-classification-sub}
\end{subfigure}
\caption{Scan-based patch construction for a smooth curved boundary.}
\label{fig:scan-to-block}
\end{figure}

The same scan-based construction can also be applied to a boundary with
small-scale local perturbations. Figure~\ref{fig:rough-scan-to-block} shows a
representative example. Although the boundary contains visible local
oscillations, the scan still produces the same types of local objects:
rectangular interior patches, one-side boundary patches of type
\(L,T,R,B\), and corner covers. Therefore the roughness does not require a
new partition logic. It is handled at the boundary-patch solver level by the
smooth-cover correction described in
Section~\ref{subsec:smooth-cover-correction}.

\begin{figure}[htbp]
\centering
\begin{subfigure}[b]{0.45\textwidth}
\centering
\includegraphics[width=\textwidth]{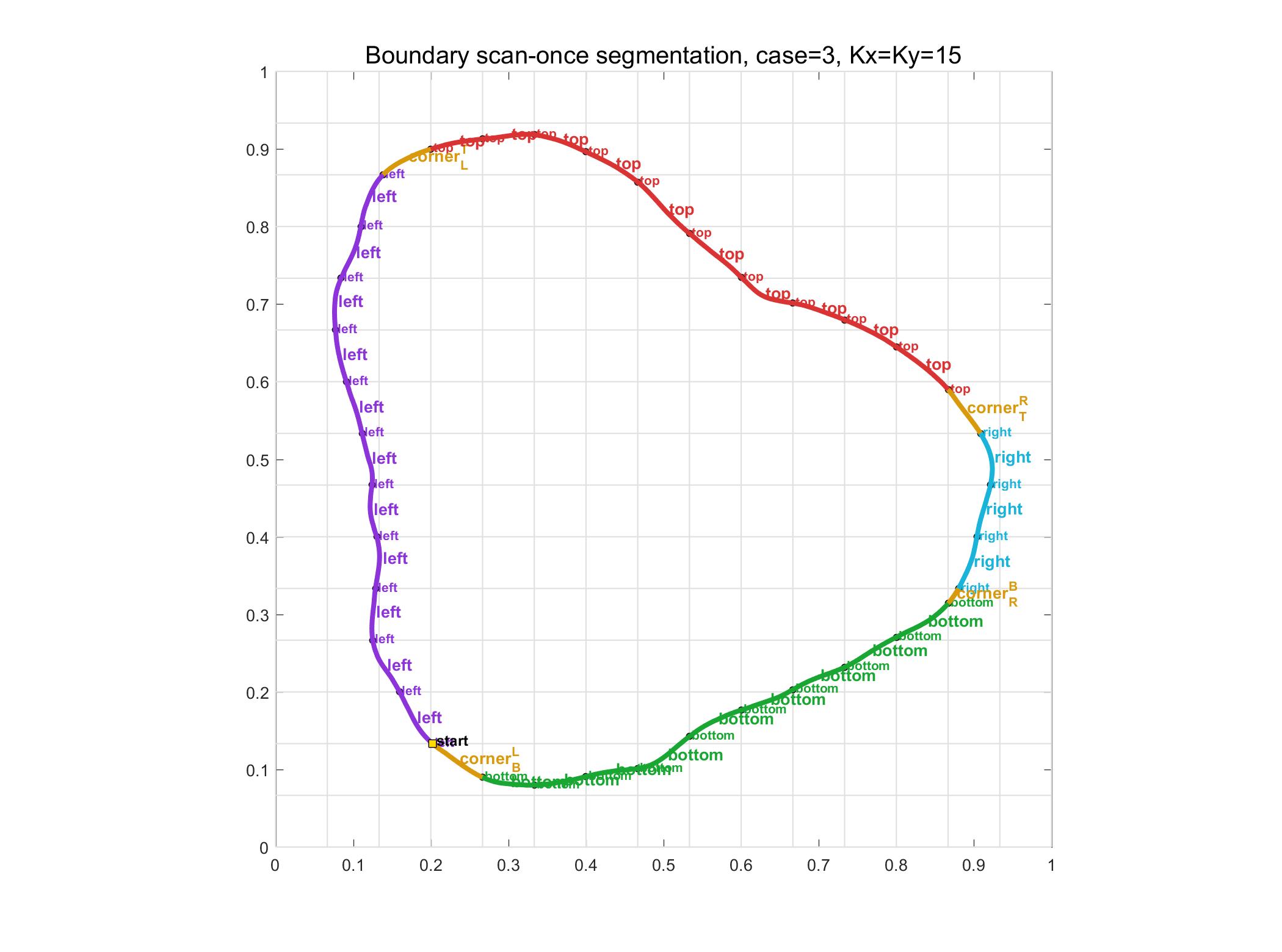}
\caption{Boundary scan and directional labeling for a boundary with
small-scale perturbations.}
\label{fig:rough-boundary-scan-sub}
\end{subfigure}
\hfill
\begin{subfigure}[b]{0.45\textwidth}
\centering
\includegraphics[width=\textwidth]{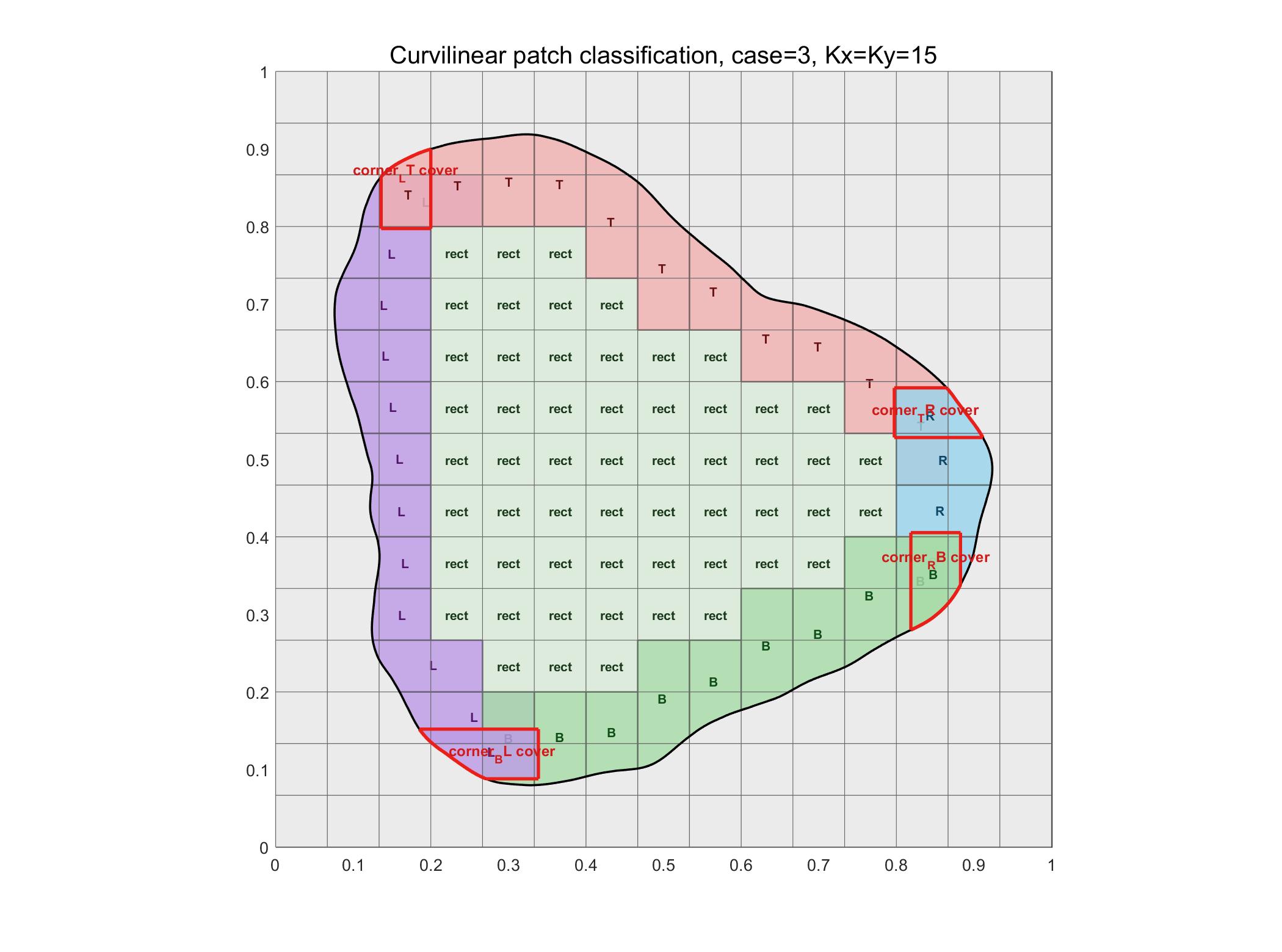}
\caption{Patch classification and corner covers for the perturbed boundary.}
\label{fig:rough-block-classification-sub}
\end{subfigure}
\caption{Scan-based patch construction for a mildly rough boundary obtained
by a small radial perturbation of a smooth reference curve. The same
classification mechanism is used as in the smooth case: interior rectangular
patches, one-side boundary patches, and corner covers are generated from the
boundary scan. }
\label{fig:rough-scan-to-block}
\end{figure}

\subsection{Computational complexity}
\label{subsec:complexity}

We briefly summarize the computational complexity of the proposed patchwise
local Fourier extension method. The detailed constructions on rectangular
patches, curved trapezoidal patches, and mildly rough boundary patches have
been described in the preceding sections. Here we emphasize the overall
scaling of the algorithm.

Let \(q=2N+1\) be the number of one-dimensional Fourier modes, \(m\) the
number of local sampling points in each coordinate direction, and
\[
    I=\#\{j:\sigma_j>\varepsilon\}
\]
the number of retained singular values in the TSVD procedure. For fixed
parameters \(m,N,T\) and a fixed refinement factor, these quantities are
moderate constants independent of the number of patches.

The offline part consists of constructing the reference one-dimensional
Fourier extension matrix and computing its truncated singular value
decomposition. Since the reference matrix has size \(m\times q\), this cost is
$O(mq^2+q^3),$
and is paid only once for a fixed set of local parameters. The same singular
systems are reused throughout the computation whenever the corresponding
local discretization is the same.

The geometric preprocessing includes the scan of the boundary curve, the
classification of local boundary directions, and the construction of the
patch database. If \(N_\Gamma\) denotes the number of discrete boundary
points and \(K_x,K_y\) are the numbers of background grid intervals, a direct
implementation costs $ O\bigl(N_\Gamma(K_x+K_y)+K_xK_y\bigr).$
This step is also performed only once for a fixed domain.

After preprocessing, the online approximation is local. Each interior
rectangular patch is handled by the tensor-product local Fourier extension
procedure. Each curved boundary patch is first converted to an \(m\times m\)
structured data array by one-dimensional LFE transfer along vertical or
horizontal sampling lines, and is then processed by the same tensor-product
procedure. For a mildly rough boundary patch, the smooth-cover correction
adds a one-unknown completion on each sampling line before the standard
curved-patch solver is applied. These boundary operations are more involved
than the rectangular-patch computation, but their sizes are still determined
only by the fixed local parameters.

Let \(N_p\) be the total number of computational patches. For fixed local
resolution, the cost of each local operation is uniformly bounded, and hence
the total online cost is
$ O(N_p).$ Since \(N_p=O(K_xK_y)\) for a fixed Cartesian background partition, and the
corresponding global fine-grid sizes satisfy
\[
    M_x=K_x(m-1)+1,\qquad
    M_y=K_y(m-1)+1,
\]
the total number of retained output points is of order
$ M_xM_y\sim K_xK_y m^2 .$
Thus, for fixed local resolution \(m\), the online computational cost is
essentially linear in the global number of retained output points.

The memory requirement is also local. The reference SVD factors are stored
once, and the patches can be processed independently. No global dense Fourier
extension matrix is formed. This locality is one of the main advantages of
the proposed method: the geometric complexity of the domain is confined to
the patch construction and local boundary transfers, while the algebraic
approximation step remains a collection of fixed-size tensor-product LFE
computations.

\section{Numerical experiments}
\label{sec:numerical}

In this section, we present numerical experiments to illustrate the accuracy
and robustness of the proposed patchwise local Fourier extension method. We
first test the method on a smooth curved domain, and then consider a domain
whose boundary contains small-scale rough perturbations. Unless otherwise
specified, the local Fourier extension parameters are fixed throughout this
section.

\subsection{Experimental setting}
\label{subsec:experimental-setting}

We first consider a curved domain \(\Omega\subset [0,1]^2\) enclosed by the
parametric boundary
\begin{equation}
\begin{aligned}
x(t) &= 0.50
      +0.38\cos t
      +0.06\cos(2t+0.6)
      -0.03\sin(3t),\\
y(t) &= 0.50
      +0.40\sin t
      -0.08\sin(2t-0.4)
      +0.03\cos(3t+0.2),
      \qquad 0\le t<2\pi .
\end{aligned}
\label{eq:test-boundary}
\end{equation}
The domain is embedded into a Cartesian background grid with
\(K_x=K_y=K\). The boundary is scanned once, and the scan is used to generate
rectangular interior patches, one-side curved boundary patches of type
\(L,T,R\), or \(B\), and additional corner-cover patches near transition
regions.

On each patch, \(m\times m\) local sampling values are used. The default
parameters are
\[
    T=4,\qquad \gamma_0=1.2,\qquad n=10,
\]
in both coordinate directions, so that
\[
    m=\lfloor \gamma_0(2n+1)\rfloor=25.
\]
The local approximation is evaluated on a refined grid with refinement factor
\(r=5\). The truncated singular value decomposition is used with a threshold
close to machine precision. When several patches generate the same physical
point in the assembled point cloud, the value computed first is retained.

For a given patch \(P\), let
\[
    \mathcal{G}_P=\{(x_\ell,y_\ell)\}_{\ell=1}^{N_P}
\]
be the refined physical evaluation grid associated with this patch, and let
\(f_h\) denote the local patchwise approximation. We define the patchwise
maximum error by
\[
    E_{\infty}(P)
    =
    \max_{1\le \ell\le N_P}
    |f_h(x_\ell,y_\ell)-f(x_\ell,y_\ell)|.
\]
For each patch type, the reported \(E_{\infty}^{\max}\) is the maximum of
\(E_\infty(P)\) over all patches of that type, while
\[
    E_{\infty}^{\rm avg}
    =
    \frac{1}{\#\mathcal P_{\rm type}}
    \sum_{P\in\mathcal P_{\rm type}}E_\infty(P)
\]
denotes the average of the patchwise maximum errors over all patches of that
type. For curved boundary patches, the refined grid follows the corresponding
curved side, so that the numerical approximation and the exact function are
compared on the same physical patch.

\subsection{Influence of the background partition on a smooth curved domain}
\label{subsec:partition-influence}

We first examine the influence of the background partition on the smooth
curved domain \eqref{eq:test-boundary}. The test function is
\[
    f(x,y)=\frac{\sin(xy)}{1+y^2}.
\]
This function is smooth and of low frequency on the computational domain, so
the test mainly reflects the effect of the curved boundary patches and the
scan-based partition.

Figure~\ref{fig:test1} shows the result for \(K_x=K_y=20\). The left panel
gives the exact function, the middle panel shows the assembled patchwise
approximation, and the right panel displays the pointwise error on a
logarithmic scale. The maximum pointwise error is about
\(5.795\times 10^{-12}\).

\begin{figure}[htbp]
\centering
\includegraphics[width=0.9\textwidth]{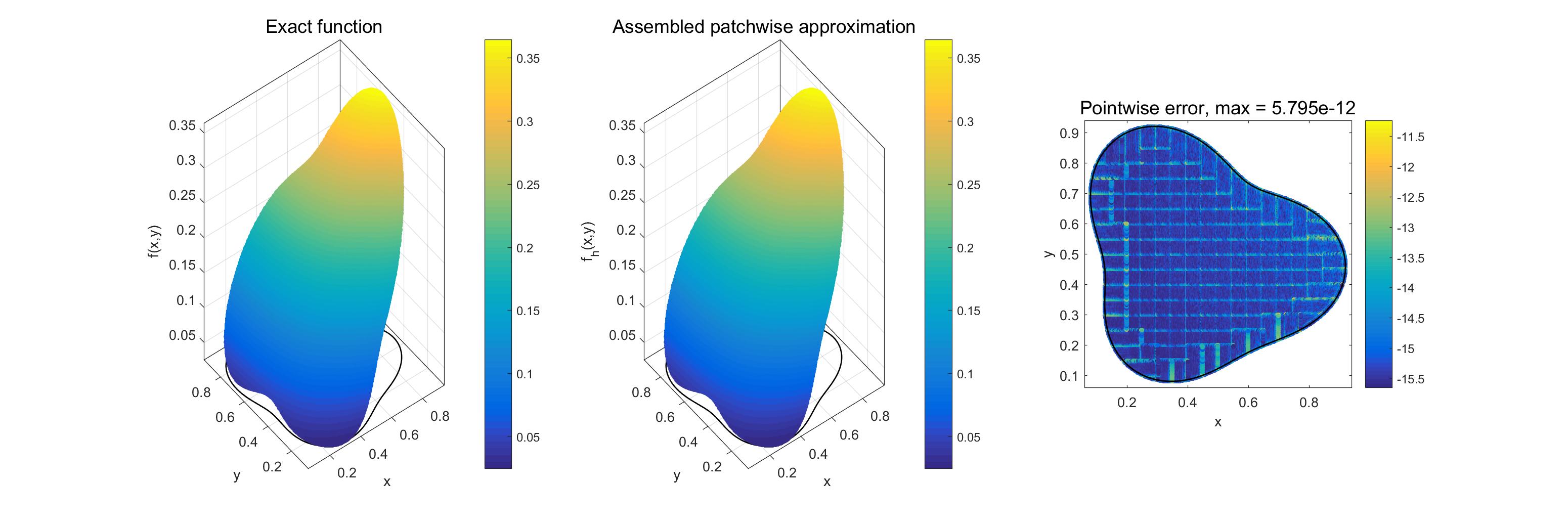}
\caption{Approximation of
\(f(x,y)=\sin(xy)/(1+y^2)\) on the smooth curved domain with
\(K_x=K_y=20\). Left: exact function. Middle: assembled patchwise
approximation. Right: pointwise error in logarithmic scale.}
\label{fig:test1}
\end{figure}

Table~\ref{tab:smooth-partition-errors} reports the patchwise maximum errors
and their averages for different background partitions. The number of patches
is also listed, since for fixed local parameters it is the main indicator of
the online cost.

The table shows that the rectangular patches already reach very high
accuracy for moderate partitions. For coarse partitions, the larger errors
are mainly associated with curved boundary patches, where a single patch may
contain a relatively long portion of the boundary. As the partition is
refined, the boundary portion represented by each patch becomes shorter and
the normalized geometric deformation is reduced. For \(K_x=K_y=20\), all
patch types reach an accuracy close to \(10^{-12}\) or better. In the
following experiments we therefore use \(K_x=K_y=20\) as the default
background partition.

\begin{table}[htbp]
	\begin{center}
		\caption{Errors for different background partitions for
        \(f(x,y)=\sin(xy)/(1+y^2)\) on the smooth curved domain.}
        \label{tab:smooth-partition-errors}
		\small\tabcolsep 0.4pt
		{\begin{tabular*}{\textwidth}{@{\extracolsep\fill}ccccccccccccc}
				\toprule
				\multirow{2}{*}{$K_x=K_y$}&\multirow{2}{*}{number of patches}
				&\multicolumn{2}{c}{\texttt{rect}}
				&\multicolumn{2}{c}{\texttt{left}}
                &\multicolumn{2}{c}{\texttt{right}}
                &\multicolumn{2}{c}{\texttt{top}}
                &\multicolumn{2}{c}{\texttt{bottom}}\\
				\cline{3-4} \cline{5-6}\cline{7-8}\cline{9-10}\cline{11-12}
				&&  $E_{\infty}^{\max}$& $E_{\infty}^{\rm avg}$
                &  $E_{\infty}^{\max}$& $E_{\infty}^{\rm avg}$
                &  $E_{\infty}^{\max}$& $E_{\infty}^{\rm avg}$
                &  $E_{\infty}^{\max}$& $E_{\infty}^{\rm avg}$
                &  $E_{\infty}^{\max}$& $E_{\infty}^{\rm avg}$\\
				\midrule
				5&14&--&--&3.91e-11&1.72e-11&2.81e-07&9.38e-08&1.42e-08&4.20e-09&6.48e-12&4.80e-12\\
                10&40&4.66e-13&2.04e-13&1.07e-09&1.36e-10&2.90e-12&2.59e-12&8.84e-10&1.27e-10&2.70e-12&1.42e-12\\
                15&92&7.40e-13&1.96e-13&2.28e-12&1.35e-12&3.05e-12&2.32e-12&4.29e-12&2.06e-12&2.92e-12&1.90e-12\\
                20&170&7.83e-13&1.91e-13&7.99e-13&3.95e-13&5.79e-12&3.16e-12&4.19e-12&1.59e-12&1.43e-12&7.08e-13\\
				\bottomrule
		\end{tabular*}}
	\end{center}
\end{table}

\subsection{Full-domain test with a mildly rough boundary}
\label{subsec:rough-boundary-test}

We next test the method on a domain whose boundary contains small-scale local
perturbations. To make the experiment reproducible, we explicitly describe
the construction of the perturbed boundary. We start from the smooth closed
reference curve
\[
    \rho_0(t)=1+0.18\cos(3t)-0.08\sin(2t),
\]
\[
    X_0(t)=\rho_0(t)\cos t,\qquad
    Y_0(t)=0.82\rho_0(t)\sin t .
\]
After affine rescaling of \((X_0,Y_0)\) into the box \([0.08,0.92]^2\), we
obtain a smooth curve \((x_0(t),y_0(t))\). The mildly rough boundary is then
defined by a small radial perturbation with respect to the center
\((0.5,0.5)\):
\[
    x_r(t)=0.5+(1+\rho_1(t))(x_0(t)-0.5),\qquad
    y_r(t)=0.5+(1+\rho_1(t))(y_0(t)-0.5),
\]
where
\[
\begin{aligned}
\rho_1(t)
={}&
0.012\,\operatorname{sgn}(\sin(12t+0.40))
|\sin(12t+0.40)|^{1.20} \\
&+
0.007\,\operatorname{sgn}(\cos(17t-0.20))
|\cos(17t-0.20)|^{1.20}
+
0.003\sin(25t+0.80).
\end{aligned}
\]
This construction preserves the overall shape of the smooth reference curve,
while introducing visible small-scale perturbations along the boundary.

The scan-based partition for this domain is shown in
Figure~\ref{fig:rough-scan-to-block}. The same patch classification is used
as in the smooth case, while the smooth-cover correction described in
Section~\ref{subsec:smooth-cover-correction} is applied to the boundary
patches of type \(L,T,R\), and \(B\).

We first use the same test function
\[
    f(x,y)=\frac{\sin(xy)}{1+y^2}.
\]
Table~\ref{tab:rough-partition-errors} reports the patchwise errors for
different background partitions.

\begin{table}[htbp]
	\begin{center}
		\caption{Errors for different background partitions for
        \(f(x,y)=\sin(xy)/(1+y^2)\) on a mildly rough boundary domain.}
        \label{tab:rough-partition-errors}
		\small\tabcolsep 0.4pt
		{\begin{tabular*}{\textwidth}{@{\extracolsep\fill}ccccccccccccc}
				\toprule
				\multirow{2}{*}{$K_x=K_y$}&\multirow{2}{*}{number of patches}
				&\multicolumn{2}{c}{\texttt{rect}}
				&\multicolumn{2}{c}{\texttt{left}}
                &\multicolumn{2}{c}{\texttt{right}}
                &\multicolumn{2}{c}{\texttt{top}}
                &\multicolumn{2}{c}{\texttt{bottom}}\\
				\cline{3-4} \cline{5-6}\cline{7-8}\cline{9-10}\cline{11-12}
				&&  $E_{\infty}^{\max}$& $E_{\infty}^{\rm avg}$
                &  $E_{\infty}^{\max}$& $E_{\infty}^{\rm avg}$
                &  $E_{\infty}^{\max}$& $E_{\infty}^{\rm avg}$
                &  $E_{\infty}^{\max}$& $E_{\infty}^{\rm avg}$
                &  $E_{\infty}^{\max}$& $E_{\infty}^{\rm avg}$\\
				\midrule
				5&14&--&--&5.34e-11&4.56e-11&7.34e-12&3.46e-12&2.50e-11&6.57e-12&1.90e-10&1.05e-10\\
                10&40&4.66e-13&2.04e-13&4.06e-11&1.48e-11&6.58e-12&4.09e-12&2.11e-12&1.13e-12&2.02e-10&7.64e-11\\
                15&93&7.40e-13&1.99e-13&2.27e-10&5.30e-11&4.06e-12&2.60e-12&6.69e-12&1.89e-12&1.03e-10&3.37e-11\\
                20&173&7.83e-13&1.90e-13&4.64e-12&1.34e-12&2.66e-10&7.44e-11&5.31e-11&2.08e-11&3.49e-12&2.04e-12\\
          	\bottomrule
		\end{tabular*}}
	\end{center}
\end{table}

With the smooth-cover correction, the boundary patch errors remain stable
despite the small-scale perturbations of the boundary. Compared with a direct
treatment of the perturbed boundary, the cover correction reduces the
sensitivity of boundary patches to local geometric oscillations. At the same
time, the partition size still affects the overall approximation through the
local frequency reduction discussed in
Remark~\ref{rem:local-frequency-reduction}. We therefore keep
\(K_x=K_y=20\) in the remaining tests on the rough boundary domain.

Figure~\ref{fig:rough-cover-comparison} compares the results obtained
without and with the smooth-cover correction for \(K_x=K_y=20\). Without the
cover correction, the maximum pointwise error is about
\(3.609\times 10^{-5}\), and the larger errors are concentrated near boundary
patches. After applying the local smooth-cover correction, the maximum
pointwise error is reduced to about \(1.298\times10^{-10}\).

Finally, we test the same rough boundary domain and the same cover-corrected
setting on four additional functions:
\[
\begin{aligned}
    f_1(x,y)&=\operatorname{erf}(20(x-y)),  & f_2(x,y)&=\frac{\log(10(x+y))}{\sqrt{x^2+y}},\\
    f_3(x,y)&=\sin\bigl(20(x^2+y^2)\bigr),    &f_4(x,y)&=\operatorname{Ai}(-15-13(x+y)),
\end{aligned}
\]
where \(\operatorname{Ai}\) denotes the Airy function. Figure~\ref{fig:rough_four_errors}
shows the corresponding results. The maximum pointwise errors are approximately
$
1.447\times 10^{-9},~
4.445\times 10^{-9},~
1.099\times 10^{-9},~
1.618\times 10^{-10},
$
respectively. These tests show that the same cover-corrected patchwise
framework remains accurate for functions with different local behaviors on
the mildly rough boundary domain.

\begin{figure}[htbp]
\centering
\begin{subfigure}[b]{0.9\textwidth}
\centering
\includegraphics[width=\textwidth]{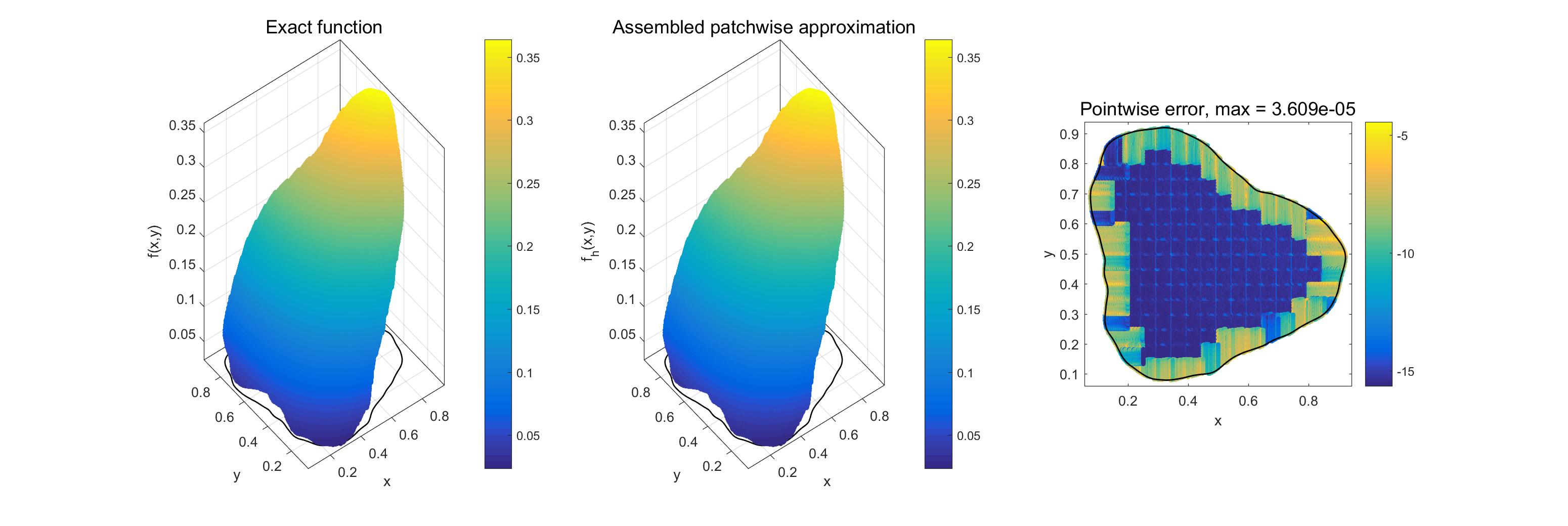}
\caption{Direct treatment of the mildly rough boundary without the
smooth-cover correction. The maximum pointwise error is
\(3.609\times10^{-5}\).}
\label{fig:rough-no-cover}
\end{subfigure}

\vspace{2mm}

\begin{subfigure}[b]{0.9\textwidth}
\centering
\includegraphics[width=\textwidth]{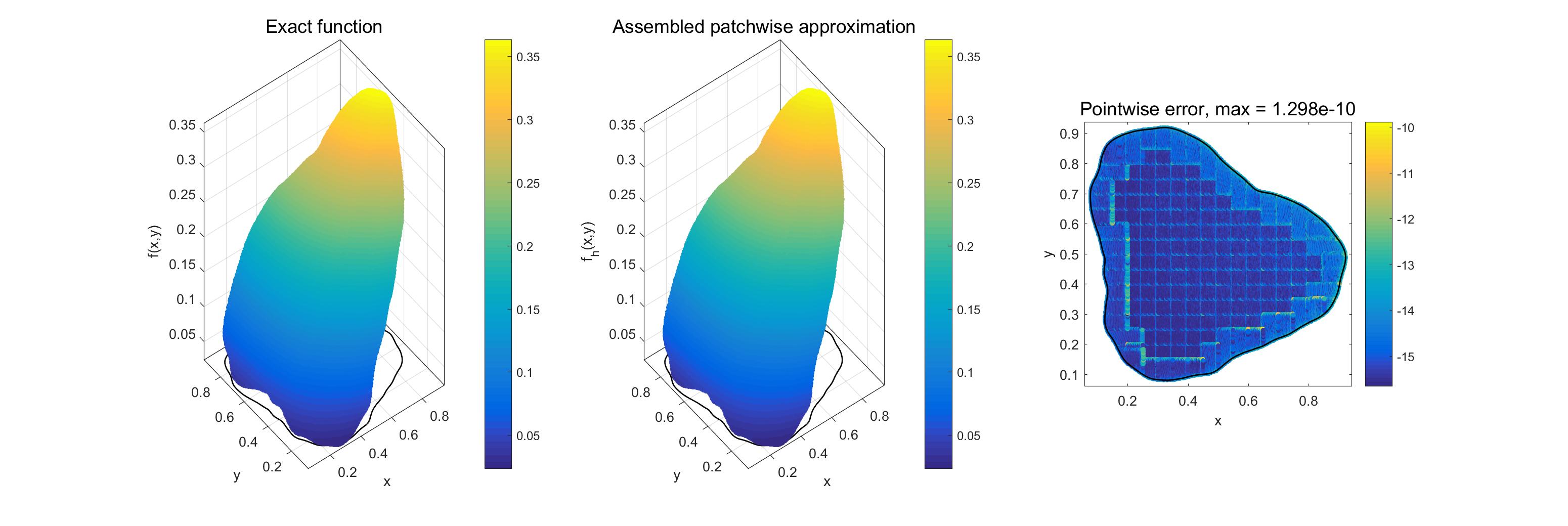}
\caption{Result obtained with the smooth-cover correction applied to
boundary patches. The maximum pointwise error is reduced to
\(1.298\times10^{-10}\).}
\label{fig:rough-with-cover}
\end{subfigure}

\caption{Effect of the smooth-cover correction on the mildly rough boundary domain for \(f(x,y)=\sin(xy)/(1+y^2)\) with \(K_x=K_y=20\).}
\label{fig:rough-cover-comparison}
\end{figure}

\begin{figure}[htbp]
\centering
\begin{subfigure}[b]{0.48\textwidth}
    \centering
    \includegraphics[width=\textwidth]{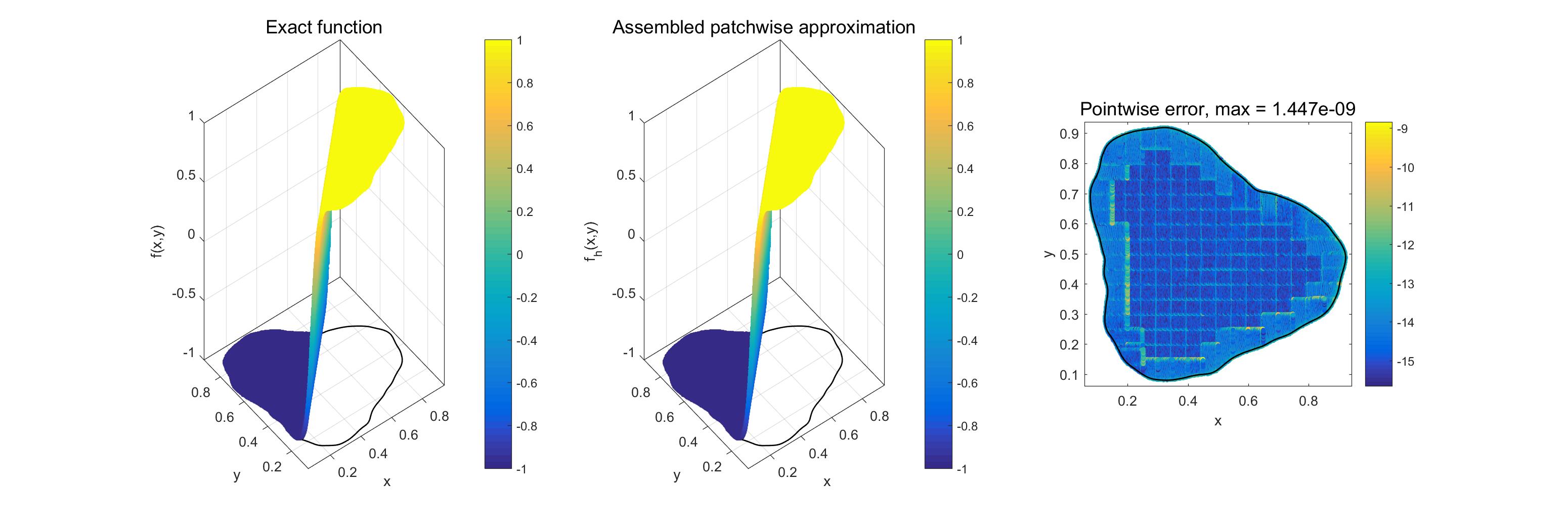}
    \caption{\(f_1(x,y)=\operatorname{erf}(20(x-y))\), max error \(=1.447\times10^{-9}\).}
    \label{fig:rough_error_erf}
\end{subfigure}
\hfill
\begin{subfigure}[b]{0.48\textwidth}
    \centering
    \includegraphics[width=\textwidth]{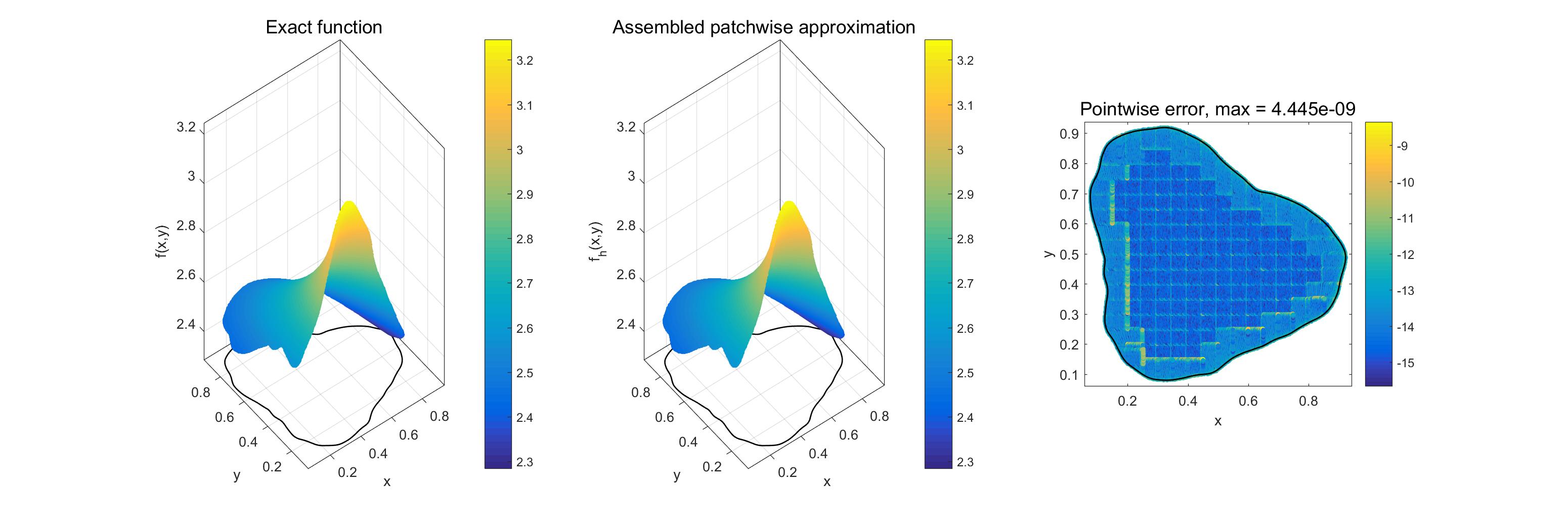}
    \caption{\(f_2(x,y)=\log(10(x+y))/\sqrt{x^2+y}\), max error \(=4.445\times10^{-9}\).}
    \label{fig:rough_error_log}
\end{subfigure}

\vspace{2mm}

\begin{subfigure}[b]{0.48\textwidth}
    \centering
    \includegraphics[width=\textwidth]{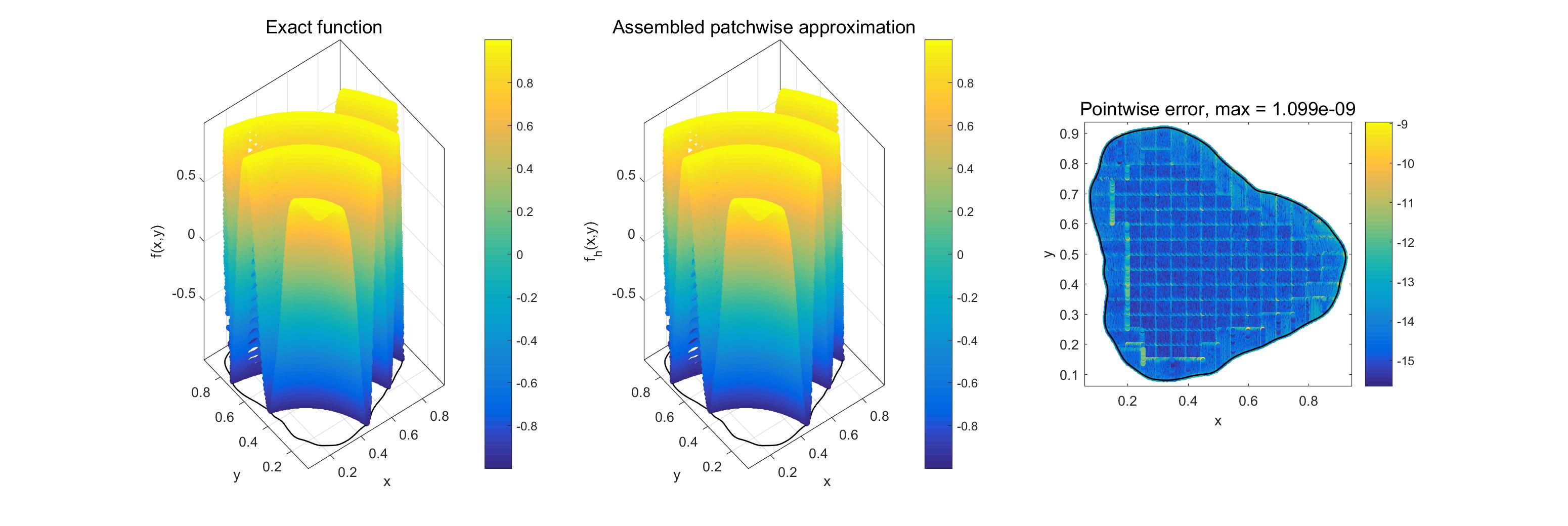}
    \caption{\(f_3(x,y)=\sin(20(x^2+y^2))\), max error \(=1.099\times10^{-9}\).}
    \label{fig:rough_error_radial}
\end{subfigure}
\hfill
\begin{subfigure}[b]{0.48\textwidth}
    \centering
    \includegraphics[width=\textwidth]{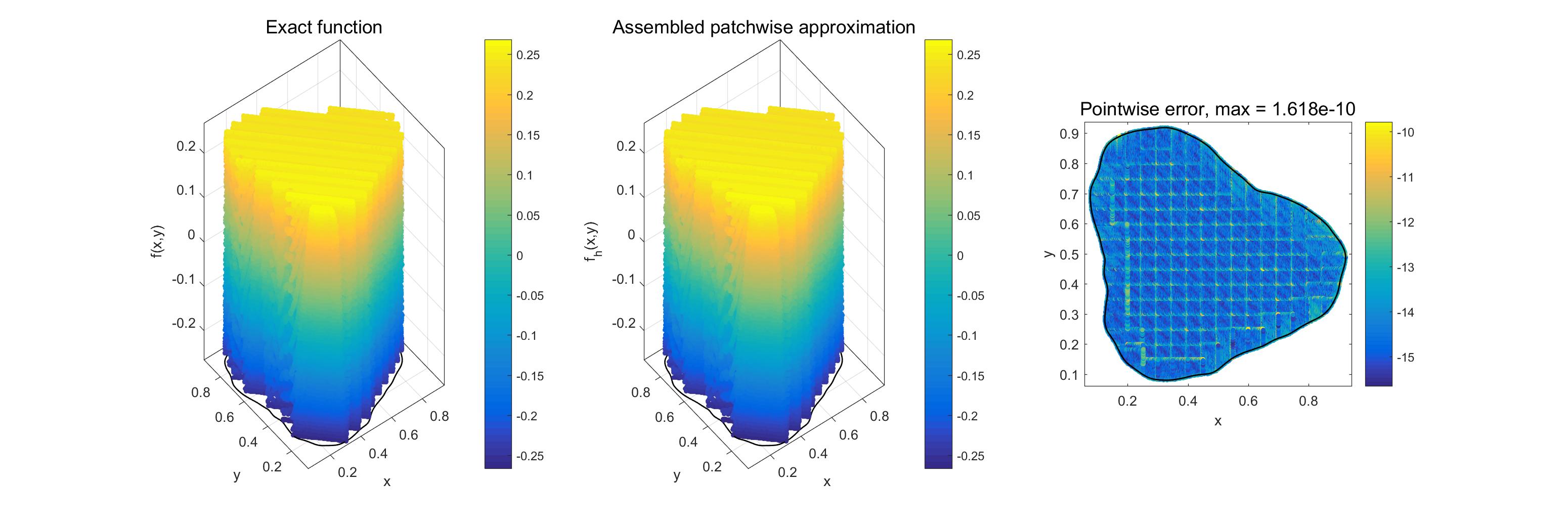}
    \caption{\(f_4(x,y)=\operatorname{Ai}(-15-13(x+y))\), max error \(=1.618\times10^{-10}\).}
    \label{fig:rough_error_airy}
\end{subfigure}

\caption{Pointwise error plots for four additional test functions on the mildly rough boundary domain with \(K_x=K_y=20\).}
\label{fig:rough_four_errors}
\end{figure}
The error plots also show that the relatively larger errors often appear near
patch interfaces or near boundary-related subregions. This is expected for a
non-overlapping patchwise assembly. A possible further improvement is to use overlapping patches and retain only
the central or better-conditioned part of each local approximation. Such an
overlap-and-restriction strategy may reduce interface effects, but we leave
this refinement for future work.

\subsection{A brief timing test}

To further illustrate the computational behavior of the proposed method, we
also report a simple timing test for the mildly rough boundary case. In this
experiment, we remove auxiliary post-processing steps such as refined-grid
evaluation and visualization, and only retain the essential geometric
preprocessing and patchwise approximation steps.
\begin{figure}[htbp]
\centering
\includegraphics[width=0.95\textwidth]{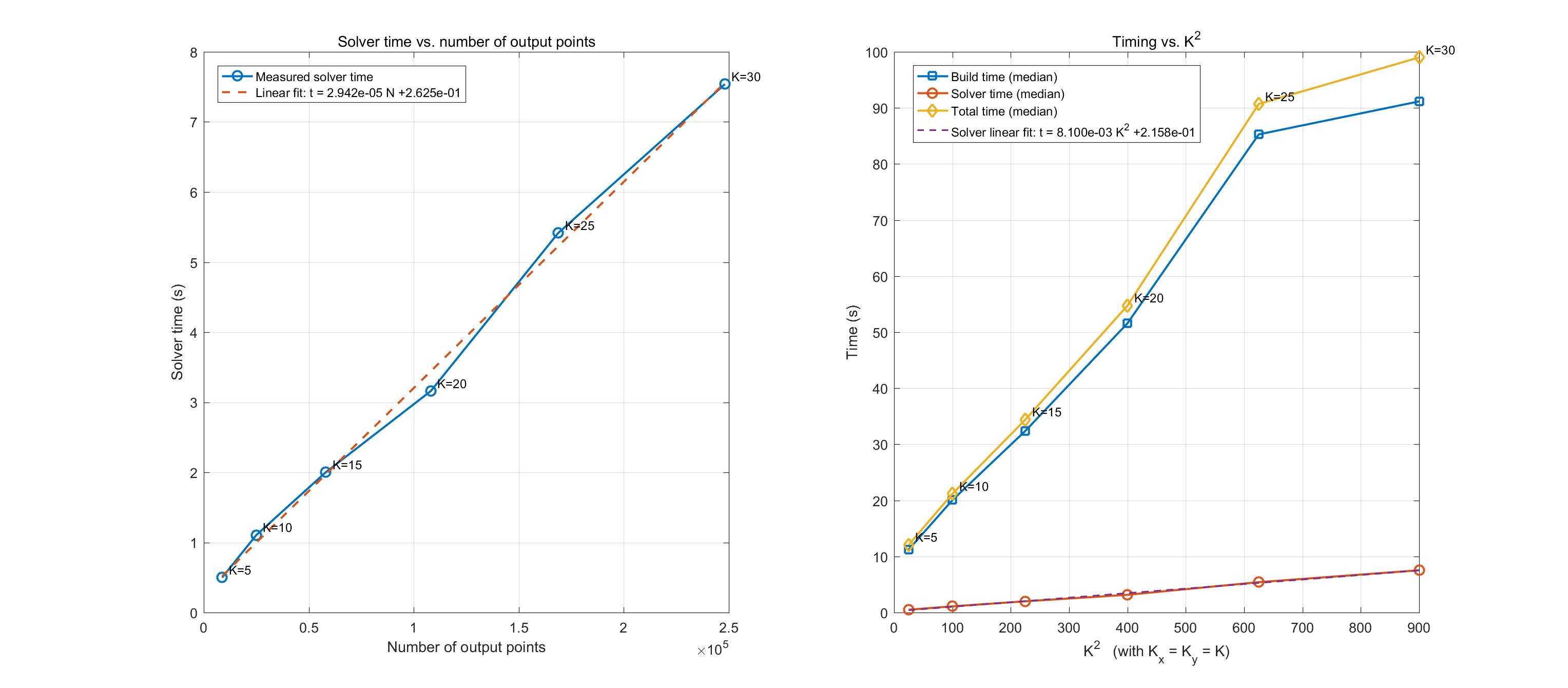}
\caption{Timing test for the patchwise solver on the mildly rough boundary domain. The left panel shows solver time versus the number of retained output points, and the right panel separates the build, solver, and total times versus \(K^2\).}
\label{fig:complexity-benchmark}
\end{figure}
Figure~\ref{fig:complexity-benchmark} reports the timing results. Since the
local parameters are fixed, the number of retained output points is
proportional to \(K_xK_y\); in the present tests \(K_x=K_y=K\), so the
expected linear online complexity corresponds to a nearly linear dependence
on \(K^2\). The left panel confirms that the solver time grows approximately
linearly with the number of retained output points. The right panel separates
the build, solver, and total times. The build stage is more expensive in this
implementation because it includes the geometric preprocessing, especially
the computation of intersections between the boundary curve and Cartesian
grid lines. This preprocessing is performed only once for a fixed domain and
can be reused when approximating different functions on the same geometry.

\section{Conclusions and remarks}
\label{sec:conclusions}

We have presented a patchwise local Fourier extension method for function
approximation on two-dimensional domains with curved boundaries. The method
embeds the physical domain into a Cartesian background grid and decomposes it
into rectangular interior patches and one-side curved boundary patches. After
local data transfer, each patch is represented by a fixed-size tensor-product
array and approximated by a TSVD-stabilized local Fourier extension procedure.
This construction avoids forming a single global Fourier frame system and
localizes both the geometric treatment and the ill-conditioned extension step.

For fixed local parameters, the local algebraic systems have bounded size and
the reference one-dimensional SVDs can be precomputed and reused. Boundary
patches require additional line transfers, and mildly rough boundary patches
also require a one-unknown completion on each sampling line, but these steps
remain local. The timing test confirms that, after the geometric
preprocessing stage, the online solver time grows approximately linearly with
the number of retained output points. The build stage is more dependent on the
boundary representation, since in the current implementation its main cost is
the computation of boundary--grid intersections.

The numerical experiments show that the method achieves high accuracy on
smooth curved domains with a fixed set of parameters. They also demonstrate
that the smooth-cover correction can substantially reduce boundary-induced
errors for mildly rough boundary patches without changing the scan-based
partition logic. The larger errors observed in some tests are mainly located
near patch interfaces or boundary-related subregions, which is consistent
with the non-overlapping assembly used in the present implementation.

Several extensions are natural. Adaptive refinement could be used to refine
only patches with large local frequency, geometric distortion, or
cover-correction residuals. An overlap-and-restriction assembly may further
reduce interface effects by retaining only the central or better-conditioned
part of each local approximation. A systematic treatment of more general
nonsmooth domains would require automatic boundary segmentation, robust cover
construction, and a stability analysis of the artificial boundary completion.
Finally, the patchwise LFE representation is a promising ingredient for
high-order PDE solvers on curved domains, either through alternating-direction
or domain-decomposition strategies based on one-dimensional LFE operators, or
through point-driven local collocation schemes.
\section*{Acknowledgments}The research is partially supported by National Natural Science Foundation of China (Nos. 42450232, 12171455).

\end{document}